\pdfoutput=1
\documentclass[DIV=12]{scrartcl}
\usepackage{amsmath,amsthm,amssymb}
\usepackage{bm}
\usepackage{mathtools}
\usepackage{stmaryrd}
\usepackage{shuffle}
\usepackage[abbrev,nobysame,lite]{amsrefs}
\usepackage[T1]{fontenc}
\usepackage{lmodern}
\usepackage[inline,shortlabels]{enumitem}
\usepackage{microtype}
\usepackage{hyperref}

\makeatletter
\DeclareOldFontCommand{\rm}{\normalfont\rmfamily}{\mathrm}
\DeclareOldFontCommand{\sf}{\normalfont\sffamily}{\mathsf}
\DeclareOldFontCommand{\bf}{\normalfont\bfseries}{\mathbf}
\DeclareOldFontCommand{\it}{\normalfont\itshape}{\mathit}

\def\blankfootnote{\xdef\@thefnmark{}\@footnotetext}

\newcommand{\onto}{\ensuremath{\twoheadrightarrow}}

\theoremstyle{plain}
\newtheorem{thm}{Theorem}[section]
\newtheorem{pro}[thm]{Proposition}
\newtheorem{lem}[thm]{Lemma}

\newtheorem{cor}[thm]{Corollary}
\theoremstyle{definition}
\newtheorem{ex}[thm]{Example}

\newtheorem{rmk}[thm]{Remark}

\newcommand{\NN}{\mathbb{N}}
\newcommand{\ZZ}{\mathbb{Z}}
\newcommand{\FF}{\mathbb{F}}
\newcommand{\QQ}{\mathbb{Q}}

\newcommand{\cA}{\mathcal{A}}

\newcommand{\bp}{{\bm p}}
\newcommand{\A}{{\mathcal A}}
\newcommand{\ve}{{\varepsilon}}

\newcommand{\dup}{\%}
\DeclarePairedDelimiter{\abs}{\lvert}{\rvert}

\DeclareMathOperator{\st}{st}
\DeclareMathOperator{\len}{sup_*}

\DeclareMathOperator{\Ker}{Ker}

\DeclareMathOperator{\Des}{Des}
\DeclareMathOperator{\sDes}{sDes}

\DeclareMathOperator{\QSym}{QSym}

\DeclareMathOperator{\comaj}{comaj}
\DeclareMathOperator{\des}{des}

\DeclareMathOperator{\Mat}{M}

\DeclareMathOperator{\ask}{ask}
\DeclareMathOperator{\cc}{cc}
\DeclareMathOperator{\oc}{oc}

\DeclareMathOperator{\palette}{{pal}}
\DeclareMathOperator{\symbols}{{sym}}

\DeclareMathOperator{\supp}{{supp}}

\DeclareMathOperator{\col}{col}

\newcommand{\bx}{{\bm x}}

\newcommand{\Zeta}{\mathsf{Z}}

\newcommand{\bcol}{\mathbf{col}}

\numberwithin{equation}{section}
\allowdisplaybreaks[1]

\renewcommand{\le}{\leqslant}
\renewcommand{\ge}{\geqslant}

\newcommand{\Asterisk}{\!\mathop{\scalebox{1.5}{\raisebox{-0.2ex}{$\ast$}}}}%
\newcommand{\BigAsterisk}{\!\mathop{\scalebox{2.5}{\raisebox{-0.2ex}{$\ast$}}}}%

\def\emph{}
\DeclareTextFontCommand{\bfemph}{\bfseries}
\DeclareTextFontCommand{\itemph}{\itshape}
\def\emph{\bfemph}

\ifpdf
  \hypersetup{pdftitle={Coloured shuffle compatibility, Hadamard products, and ask zeta functions},
    pdfauthor={Angela Carnevale, Vassilis Dionyssis Moustakas, Tobias Rossmann}
  }
\fi

\title{Coloured shuffle compatibility, Hadamard products, and ask zeta functions}
\author{%
  \href{mailto:angela.carnevale@universityofgalway.ie}{Angela Carnevale}%
  \footnote{School of Mathematical and Statistical Sciences, University of Galway, Ireland}
  \and
  \href{mailto:vasileios.moustakas@dm.unipi.it}{Vassilis Dionyssis Moustakas}%
  \footnote{Dipartimento di Matematica, Universit\`a di Pisa, Italy}
  \and
  \setcounter{footnote}{1}
  \href{mailto:tobias.rossmann@universityofgalway.ie}{Tobias Rossmann}${}^{\fnsymbol{footnote}}$}
\date{}

\begin{document}

\maketitle

%%%%%%%%%%%%%%%%%%%%%%%%%%%%%%%%%%%%%%%%%%%%%%%%
\begin{abstract}
%%%%%%%%%%%%%%%%%%%%%%%%%%%%%%%%%%%%%%%%%%%%%%%% 
  We devise an explicit method for computing
  combinatorial formulae for Hadamard products of
  certain rational generating functions.
  The latter arise naturally when studying so-called ask zeta functions of direct sums of
  modules of matrices or class- and orbit-counting zeta functions of
  direct products of nilpotent groups.
  Our method relies on shuffle compatibility of coloured permutation
  statistics and coloured quasisymmetric functions, extending recent work of
  Gessel and Zhuang.
\end{abstract}

%%%%%%%%%%%%%%%%%%%%%%%%%%%%%%%%%%%%%%%%%%%%%%%%
\blankfootnote{\\\noindent{\itshape 2020 Mathematics Subject Classification.}
%%%%%%%%%%%%%%%%%%%%%%%%%%%%%%%%%%%%%%%%%%%%%%%%
  05A15, % Exact enumeration problems, generating functions
  11M41, % Other Dirichlet series and zeta functions
  20D15, % Nilpotent groups, $p$-groups
  20E45, % Conjugacy classes
  05E05, % Symmetric functions and generalizations
  15B33  % Matrices over special rings (quaternions, finite fields, etc.)

}

%%%%%%%%%%%%%%%%%%%%%%%%%%%%%%%%%%%%%%%%%%%%%%%%%%%%%%%%%%%%%%%%%%%%%%%%
\section{Introduction}\label{sec:introduction}
%%%%%%%%%%%%%%%%%%%%%%%%%%%%%%%%%%%%%%%%%%%%%%%%%%%%%%%%%%%%%%%%%%%%%%%%

\paragraph{Permutation statistics and shuffle compatibility.}
Permutation statistics are functions defined on permutations and their
generalisations. Studying the behaviour of said functions on sets of
permutations is a classical theme in algebraic and enumerative
combinatorics.
The origins of permutation statistics can be traced back to work of
Euler and MacMahon.
The past decades saw a flurry of further developments in the area;
see e.g.\ \cite{reiner93, Bre94, steingr94, Petersen, FZ/90, bjornerwachs91,
  BjBr/05,Stanley/97}.

Recently, Gessel and Zhuang~\cite{GZ18} developed an algebraic framework for
systematically studying what they dubbed \itemph{shuffle-compatible}
permutation statistics by means of associated shuffle algebras.
In their work, \itemph{quasisymmetric functions} (and the closely related
$P$-partitions) as well as \itemph{Hadamard products} of rational
generating functions played key roles.
For further developments building upon \cite{GZ18}, see e.g.\ \cite{Gri18}.

\paragraph{Zeta functions and enumerative algebra.}
Numerous types of \itemph{zeta functions} have been employed in the study of
enumerative problems surrounding algebraic structures.
L.\ Solomon~\cite{Sol77} introduced zeta functions associated
with integral representations.
In another influential paper, Grunewald, Segal, and Smith~\cite{GSS/88}
initiated the study of zeta functions associated with (nilpotent and pro-$p$)
groups.
Among the counting problems considered in \cite{GSS/88} are the enumeration of
subgroups and normal subgroups of suitable groups according to their
indices.
The associated zeta functions are then Dirichlet generating functions
encoding, say, numbers of subgroups of given index.

Over the following decades, a variety of methods have
been developed and applied to predict the behaviour and to study symmetries
of zeta functions associated with algebraic structures, and to produce
explicit formulae.
Theoretical work in this area often employs a blend of combinatorics and
$p$-adic integration; see \cite{Vol11} for a survey.
On the practical side, a range of effective methods have
been devised, implemented (see \cite{zeta}), and used to symbolically compute zeta functions of algebraic
structures; see \cite{Ros17/1489} and the references therein.

\paragraph{Variation of the prime: symbolic enumeration.}
A common feature of zeta functions $\zeta_G(s)$ attached to an algebraic
structure $G$ in the literature is that they often admit an
Euler factorisation $\zeta_G(s) = \prod_p \zeta_{G,p}(s)$ into
\itemph{local factors} indexed by primes $p$.
For instance, when $G$ is a nilpotent group and we are
enumerating its subgroups, we can view $\zeta_{G,p}(s)$ as the ordinary
generating function in $p^{-s}$ encoding the numbers of subgroups of
$p$-power index of $G$.

Results from $p$-adic integration often guarantee that local
factors $\zeta_{G,p}(s)$ are rational in $p^{-s}$, i.e.\ of the form $\zeta_{G,p}(s) =
  W_p(p^{-s})$ for some $W_p(Y) \in \QQ(Y)$;
see, in particular, \cite[Thm~1]{GSS/88}.
A key theme is then to study how the $W_p(Y)$ vary with the
prime $p$.
By the \itemph{symbolic enumeration} of the objects counted by $\zeta_G(s)$,
we mean the task of providing a meaningful description of $W_{p}(Y)$ as
a function of $p$.
In a surprising number of cases of interest, deep \itemph{uniformity results}
ensure the existence of a \itemph{single} bivariate rational function $W(X,Y)$ such that $\zeta_{G,p}(s) =
  W(p,p^{-s})$ for all primes $p$ (perhaps ignoring a finite number of
exceptions);
see, for example,
\cite[Thm~1.1]{BGS/21},
\cite[Thm~1.2]{CSV/19}, \cite[Thm~A]{cico}, or
\cite[Thm~B]{StasinskiVoll/14}.
In such situations, understanding our zeta function is tantamount
to understanding $W(X,Y)$.

In this context, permutation statistics (and, more generally, combinatorial
objects such as graphs and posets) have recently found spectacular
applications, in particular when it comes to describing the numerators of the
rational functions $W(X,Y)$ from above; see, for instance,
\cite{BGS/21,CarnevaleVoll/18,CSV/19,CSV_fpsac,CSV/18,StasinskiVoll/14}.
Conversely, the need for combinatorial descriptions of such zeta functions gave
rise to new directions in the study of permutation statistics and,
more generally, combinatorial objects; see, e.g.~\cite{BC1,Carnevale/17,
  CT,Denert/90,FZ/90,DMS/23,StasinskiVoll/13}.

\paragraph{Average sizes of kernels and orbits of groups.}
Introduced in \cite{Rossmann/18} and developed further in \cite{ask2,cico,board},
\itemph{ask zeta functions} are generating functions encoding average
sizes of kernels in modules of matrices.
One main motivation for studying these functions comes from group theory.
Indeed, for groups with a sufficiently powerful Lie theory, the enumeration
of linear orbits and conjugacy classes boils down to determining average sizes
of kernels within matrix
Lie algebras---this is essentially the orbit-counting lemma.
For details, see \cite[Prop.\ 8.13]{Rossmann/18}, \cite[\S 8.5]{Rossmann/18}, and
\cite[Prop.\ 6.4]{ask2}.

Amidst a plethora of algebraically-defined zeta functions,
ask zeta functions stand out as particularly amenable to combinatorial
methods.
Indeed, natural operations at the level of the modules (or groups)
often translate into natural operations of corresponding rational
generating functions.
In particular, ask zeta functions of (diagonal) direct sums of modules are
Hadamard products of the ask zeta function of the summands;
see \cite[\S 3.4]{Rossmann/18}.
\\

\noindent
In the present paper, we employ algebro-combinatorial tools such as shuffle
algebras and permutation statistics to answer questions pertaining to ask zeta
functions.
In this way, our work bridges two mathematical areas, namely algebraic
combinatorics and enumerative algebra.
Our main result belongs to  algebraic combinatorics.
Its applications to enumerative algebra are obtained by means of a method
for explicitly computing certain Hadamard products and this method acts as a
bridge between the two fields.

\paragraph{Main result.}
A first, somewhat coincidental, application of permutation statistics in the
context of ask zeta functions appeared in \cite[\S 5.4]{Rossmann/18};
see also Remark~\ref{r:brenti}.
Here, we develop this connection significantly further.
In Sections~\ref{s:qsym}--\ref{s:coloured},
we introduce shuffle algebras attached to coloured descent statistics (based
on \cite{Moustakas/th}) and we relate them to Poirier's coloured quasisymmetric
functions~\cite{Poirier}.
Along the way, in Section~\ref{ss:A_des_comaj_p},
we consider a triple of coloured permutation statistics which simultaneously
keeps track of (coloured) descent numbers, (coloured) comajor indices, and
colour multiplicities---this triple turns out to be shuffle compatible.
Drawing upon the work of Gessel and Zhuang~\cite{GZ18},
in Theorem~\ref{thm:triple_compatibility}, we obtain an explicit
embedding of the shuffle algebra associated with our triple into a power
series ring with multiplication given by the Hadamard
product.
We consider this structural result about shuffle algebras
attached to coloured permutation statistics as the main contribution
of this paper to algebraic combinatorics.

\paragraph{A method.}
In Section~\ref{elementary}, we state Theorem~\ref{thm:main_elementary},
a self-contained and elementary form of
Theorem~\ref{thm:triple_compatibility} which acts as our bridge
between algebraic combinatorics and enumerative algebra.
Designed with a view towards applications to ask zeta functions,
this result provides a method for computing explicit
formulae for Hadamard products of certain rational functions attached to
combinatorial data (``labelled coloured configuration'').
In Section~\ref{ss:proof_elementary},
we will deduce Theorem~\ref{thm:main_elementary}
from Theorem~\ref{thm:triple_compatibility}.

\paragraph{Applications.}
In Section~\ref{s:appzeta}, we conclude our paper with applications to zeta
functions.  We use our method to produce explicit combinatorial
formulae for Hadamard products within natural (infinite) families of
zeta functions.  In particular, Corollary~\ref{cor:Hadamard_Mde}
completely settles a question from \cite{cico} which asked for an
interpretation of certain ask zeta functions attached to hypergraphs
in terms of permutation statistics.  As our main group-theoretic
applications, we then obtain explicit formulae for zeta functions
enumerating
\begin{enumerate}[(a)]
  \item conjugacy classes in groups derived from direct products of free
        class-$2$-nilpotent groups (Corollary~\ref{cor:hada_F2d}) and
      \item linear orbits of direct products of full unitriangular
        matrix groups (Corollary~\ref{cor:hada_Ud}).
\end{enumerate}

\vspace*{1em}

\noindent
An extended abstract \cite{fpsac} of the present article was published in the
proceedings of the conference FPSAC 2024.
Some results in Sections \ref{ss:shuffles} and \ref{ss:A_des_comaj_p}
follow (otherwise unpublished) results from the second
author's PhD thesis \cite{Moustakas/th}.

\vspace*{1em}

\noindent
Throughout this article, all rings and algebras will be assumed to be
commutative and unital.

%%%%%%%%%%%%%%%%%%%%%%%%%%%%%%%%%%%%%%%%%%%%%%%%%%%%%%%%%%%%%%%%%%%%%%%%
\section{Hadamard products and coloured configurations}
\label{elementary}
%%%%%%%%%%%%%%%%%%%%%%%%%%%%%%%%%%%%%%%%%%%%%%%%%%%%%%%%%%%%%%%%%%%%%%%%

In this section, we provide relevant concepts and notation  and we provide a
self-contained account of our main result pertaining to Hadamard products of
suitable rational generating functions.
Its proof relies on the coloured shuffle compatibility of certain permutation
statistics and the structure of associated coloured shuffle algebras.
We will describe the latter in Section~\ref{s:coloured}.

%%%%%%%%%%%%%%%%%%%%%%%%%%%%%%%%%%%%%%%%%%%%%%%%
\subsection{Coloured permutations}
\label{ss:colperm}
%%%%%%%%%%%%%%%%%%%%%%%%%%%%%%%%%%%%%%%%%%%%%%%%

\paragraph{Coloured elements, sets, and integers.}
We consider the poset $\Gamma = \{ 0 > 1 > 2 > \dotsb\}$, the elements of
which we call \emph{colours}.
By a \emph{coloured element} of a set $T$, we mean an expression $t^c$
(formally: a pair $(t,c)$) for $t\in T$ and $c\in\Gamma$.
By a \emph{coloured subset} of $T$, we mean a set of the form
$\{t^{\gamma(t)} : t\in S\}$ for a subset $S \subseteq T$ and a function
$\gamma\colon S \to \Gamma$.

Let $\Sigma = \{ 1 < 2 < \dotsb\}$ be the usual poset of positive integers.
On the set of all coloured positive integers, we consider the total order
$$\dotsb <1^{1}<2^{1}<\dotsb<1^0<2^0<\dotsb.$$
That is, $\sigma_1^{\gamma_1}<\sigma_2^{\gamma_2}$ if and only if $\gamma_1=\gamma_2$ and
$\sigma_1<\sigma_2$, or if $\gamma_1>\gamma_2$ in $\ZZ$ (equivalently:
$\gamma_1 < \gamma_2$ in $\Gamma$).
This is the \emph{colour order} (see \cite[\S 3]{BC12}) and it corresponds to
the left lexicographic order on $\Gamma\times \Sigma$.

\paragraph{Coloured permutations and descents.}
By a \emph{coloured permutation}
we mean a string $\bm a = \sigma^\gamma = \sigma_1^{\gamma_1} \dotsb
\sigma_n^{\gamma_n}$ for $n\ge 0$, distinct $\sigma_1,\dotsc,\sigma_n\in \Sigma$,
and arbitrary $\gamma_1,\dotsc,\gamma_n\in \Gamma$.
We write $\abs{\bm a} = n$ for the \emph{length} of $\bm a$.
(There is a unique coloured permutation of length zero.)
We further write $\symbols(\bm a) = \{ \sigma_1,\dotsc,\sigma_n\}$ and
$\palette(\bm a) = \{ \gamma_1,\dotsc,\gamma_n\}$
for the \emph{set of symbols} and \emph{palette} of $\bm a$, respectively.
Finally, let
$\palette^*(\bm a) = \palette(\bm a) \setminus \{0\}$.
For instance, $\bm a = 5^12^23^01^0$ is a coloured permutation of length $4$
with $\symbols(\bm a)=\{1,2,3,5\}$ and $\palette^*(\bm a)=\{1,2\}$.
Note that our notion of coloured permutations generalises the group-theoretic
one which corresponds to the case $\symbols(\bm a) = \{ 1,\dotsc,\abs{\bm a} \}$.
  
As usual, we write $[m]=\{1,\dotsc,m\}$. 
The \emph{descent set} of a coloured permutation $\bm a = \sigma^\gamma$ as above
is
$$\Des(\bm a)=\begin{cases} \{i\in[n-1]:\sigma_i^{\gamma_i}>\sigma_{i+1}^{\gamma_{i+1}}\}, & \text{
              if } \gamma_1=0,                                                                    \\
              \{i\in[n-1]:\sigma_i^{\gamma_i}>\sigma_{i+1}^{\gamma_{i+1}}\}\cup
              \{0\},                                                         & \text{ otherwise.}\end{cases}$$
Here, the comparisons involve the colour order defined above.
The \emph{descent number} and \emph{comajor index} are defined (as usual) as
functions of the descent set: $\des(\bm a)=\abs{\Des(\bm a)}$ and
$\comaj(\bm a)=\sum_{i\in\Des(\bm a)}(n-i)$.
For $\bm a = 5^12^23^01^0$ as before, we then have $\Des(\bm a)=\{0,1,3\}$,
$\des(\bm a)=3$, and $\comaj(\bm a)=8$.

%%%%%%%%%%%%%%%%%%%%%%%%%%%%%%%%%%%%%%%%%%%%%%%%
\subsection{Labelled coloured configurations}
\label{ss:lcc}
%%%%%%%%%%%%%%%%%%%%%%%%%%%%%%%%%%%%%%%%%%%%%%%% 

\paragraph{Coloured configurations.}
Let $\cA$ be the set of all coloured permutations, and
let $\NN_0\cA$ be the free commutative monoid with basis $\cA$.
We call elements of $\NN_0\cA$ \emph{coloured configurations}.
These elements are of the form
$f = \sum_{\bm a\in \cA} f_{\bm a} \bm a$, where each $f_{\bm a}$ belongs to
$\NN_0$ and almost all $f_{\bm a}$ are zero.
(Hence, coloured configurations and multisets of coloured permutations are
identical concepts.)
Write $\supp(f) = \{ \bm a\in \cA : f_{\bm a}\neq 0\}$ and $\symbols(f) =
\bigcup_{\bm a\in \supp(f)} \symbols(\bm a)$.
Further let $\palette^*(f) = \bigcup_{\bm a \in \supp(f)} \palette^*(\bm a)$.
We call $f,g\in \NN_0\cA$
\emph{(symbol-)disjoint} if $\symbols(f) \cap \symbols(g) = \varnothing$;
if, in addition, $\palette^*(f) \cap \palette^*(g) = \varnothing$, then $f$
and $g$ are \emph{strongly disjoint}.

Let $\bm a,\bm b\in \cA$ be disjoint.
A \emph{shuffle} of $\bm a$ and $\bm b$ is a coloured permutation of length
$\abs{\bm a} + \abs{\bm b}$ which contains both $\bm a$ and $\bm b$ as not
necessarily contiguous substrings. 
Let $\bm a\shuffle \bm b\in \NN_0\cA$ be the sum over all shuffles of $\bm a$
and $\bm b$.
(Since $\bm a$ and $\bm b$ are disjoint, these shuffles are themselves
coloured permutations.)
We bi-additively extend this to define $f\shuffle g$ for disjoint $f,g\in
  \NN_0\cA$;
observe that $\symbols(f\shuffle g) = \symbols(f) \cup \symbols(g)$
and $\palette(f\shuffle g) = \palette(f) \cup \palette(g)$.
(Hence, $\palette^*(f\shuffle g) = \palette^*(f) \cup \palette^*(g)$.)

\paragraph{Labels.}
Let $\mathbb U = \{ \pm X^{k} : k\in \ZZ\}$, viewed as a subgroup of the multiplicative group
of the field $\QQ(X)$.
For $\alpha\colon \Gamma\to \mathbb U$, write $\supp(\alpha) = \{ c\in \Gamma :
  \alpha(c)\neq 1\}$ and, for $\bm a = \sigma_1^{\gamma_1} \dotsb
  \sigma_n^{\gamma_n}$ as above, let $\alpha(\bm a) = \prod_{i=1}^n
  \alpha(\gamma_i)$.
A \emph{labelled coloured configuration} is a pair $(f,\alpha)$, where $f\in
  \NN_0\cA$ and $\alpha\colon \Gamma\to \mathbb U$ satisfies $\supp(\alpha) \subseteq
  \palette^*(f)$.

\paragraph{Equivalence.}
Let $(f,\alpha)$ be a labelled coloured configuration.
Let $\phi\colon \symbols(f) \to S$ and $\psi\colon \palette^*(f) \to P$
be order-preserving bijections onto finite subsets of $\Sigma$ and
$\Gamma\setminus\{0\}$, respectively.
Given $\phi$ and $\psi$, define a labelled coloured configuration
$(f',\alpha')$ as follows.
For $\bm a\in \supp(f)$, say $\bm a = \sigma_1^{\gamma_1}\dotsb
  \sigma_n^{\gamma_n}$, write $\bm a' = \phi(\sigma_1)^{\psi(\gamma_1)}\dotsb
  \phi(\sigma_n)^{\psi(\gamma_n)}$,
where $\psi(0) := 0$.
Define $f' = \sum_{\bm a\in \supp(f)}f_{\bm a}\bm a'$.
We define $\alpha'$ to be the function $\Gamma \to \mathbb U$
whose support is contained in $P$ and which satisfies $\alpha'(\psi(c))
  = \alpha(c)$ for $c\in \palette^*(f)$.
We call $(f,\alpha)$ and each $(f',\alpha')$ (as $\phi$ and $\psi$ range
over possible choices) \emph{equivalent}, written $(f,\alpha) \asymp
(f',\alpha')$.
Informally, two labelled coloured configurations are
equivalent if one can be obtained from the other by renaming both its
symbols and its nonzero colours in an order-preserving manner.
This defines an equivalence relation on labelled coloured configurations.

\paragraph{Coherence.}
We say that labelled coloured configurations $(f,\alpha)$ and $(g,\beta)$ are
\emph{coherent} if $f$ and $g$ are disjoint and if, in addition, $\alpha(c) =
  \beta(c)$ for all $c\in \palette^*(f)\cap \palette^*(g)$.
In that case, we may define $\alpha\cup\beta\colon
  \Gamma \to \mathbb U$ via
$$
  (\alpha\cup\beta)(c) = \begin{cases}
    \alpha(c), & \text{if } c\in \palette^*(f), \\
    \beta(c),  & \text{if } c\in \palette^*(g), \\
    1,         & \text{otherwise.}
  \end{cases}
$$%
The pair $(f\shuffle g, \alpha\cup\beta)$ is then a labelled coloured configuration too.
(Indeed, we have
$\supp(\alpha\cup \beta) = \supp(\alpha)\cup \supp(\beta)
  \subseteq \palette^*(f)\cup \palette^*(g) = \palette^*(f\shuffle g)$.)

Note that if $(f,\alpha)$ and $(g,\beta)$ are labelled coloured configurations
such that $f$ and $g$ are strongly disjoint, then $(f,\alpha)$ and
$(g,\beta)$ are automatically coherent and, moreover, $\alpha\cup\beta = \alpha\beta$
is the pointwise product of $\alpha$ and $\beta$.

%%%%%%%%%%%%%%%%%%%%%%%%%%%%%%%%%%%%%%%%%%%%%%%%
\subsection{From labelled coloured configurations to Hadamard products}
%%%%%%%%%%%%%%%%%%%%%%%%%%%%%%%%%%%%%%%%%%%%%%%%

Given a labelled coloured configuration $(f,\alpha)$ and $\ve \in \ZZ$, we define a rational formal power series
\[
  W_{f,\alpha}^\ve = W_{f,\alpha}^\ve(X,Y) = \sum_{\bm a \in \supp(f)}
  f_{\bm a} \frac{\alpha(\bm a)X^{\ve\comaj(\bm a)} Y^{\des(\bm a)} }
  {(1-Y)(1-X^{\ve} Y)\dotsb(1-X^{\ve\lvert \bm a\rvert}Y)}
  \in \QQ(X)\llbracket Y\rrbracket .
\]
Note that, by construction, if $(f,\alpha)\asymp (f',\alpha')$, then
$W_{f,\alpha}^\ve = W_{f',\alpha'}^\ve$ for all $\ve \in \ZZ$.
Indeed, for $\bm a \in \supp(f)$ corresponding to $\bm a' \in \supp(f')$ as in
the definition of equivalence in \S\ref{ss:lcc}, we have
$\Des(\bm a) = \Des(\bm a')$ and $\alpha(\bm a) = \alpha'(\bm a')$.

\begin{ex}
  \label{ex:1^0+1^1}
  Let $f = 1^0 + 1^1$.
  Let $\alpha\colon \Gamma\to \mathbb U$ with $\supp(\alpha) \subseteq\palette^*(f) =
    \{1\}$.
  Then, by definition, $\alpha(0) = 1$ and we obtain
  \[
    W_{f,\alpha}^\ve = \frac{1 + \alpha(1) X^\ve Y}{(1-Y)(1-X^\ve Y)}.
  \]
\end{ex}

Recall that the \emph{Hadamard product} of two formal power series $A(Y) =
  \sum_{k=0}^\infty a_k Y^k$ and $B(Y) = \sum_{k=0}^\infty b_k Y^k$
(with coefficients in a common field, say) is
the power series $A(Y) *_Y B(Y) = \sum_{k=0}^\infty a_kb_k Y^k$.
It is well known that the Hadamard product of two rational formal power series
is itself rational; see e.g.\ \cite[Ch.\ 1]{BR88} or
\cite[Prop.\ 4.2.5]{Stanley/97}.
Despite the simple definition of this operation, explicitly computing closed
formulae for Hadamard products of rational generating functions appears to be
difficult.

The following is the key theoretical result underpinning our method for
explicitly computing Hadamard products of ask zeta functions.

\begin{thm}
  \label{thm:main_elementary}
  Let $(f,\alpha)$ and $(g,\beta)$ be coherent labelled coloured
  configurations.
  Then
  $$W_{f,\alpha}^\ve *_Y W_{g,\beta}^\ve = W_{f\shuffle g, \, \alpha\cup\beta}^\ve$$
  for each $\ve\in \ZZ$.
\end{thm}

We will prove Theorem~\ref{thm:main_elementary} in Section~\ref{ss:proof_elementary}.

\begin{ex}
  \label{ex:12}
  Let $f = 1^0 + 1^1$ and $g = 2^0 + 2^2$.
  Then
  \begin{align*}
    f\shuffle g & = (1^0 + 1^1)\shuffle (2^0 + 2^2)                                       \\
                & = 1^0 \shuffle 2^0 + 1^0 \shuffle 2^2 + 1^1 \shuffle 2^0 + 1^1 \shuffle
    2^2                                                                                   \\
                & = 1^0 2^0 + 2^0 1^0 + 1^0 2^2 + 2^2 1^0 + 1^1 2^0 + 2^0 1^1 + 1^1 2^2 +
    2^2 1^1.
  \end{align*}

  The descent numbers and comajor indices of the coloured permutations
  appearing as summands above are recorded in the following table:
  \begin{center}\begin{tabular}{l|cccccccc}
               & $1^0 2^0$ & $2^0 1^0$ & $1^0 2^2$ & $2^2 1^0$ & $1^1 2^0$ & $2^0 1^1$ & $1^1 2^2$ & $2^2 1^1$ \\
      \hline
      $\des$   & 0         & 1         & 1         & 1         & 1         & 1         & 2         & 1         \\
      $\comaj$ & 0         & 1         & 1         & 2         & 2         & 1         & 3         & 2
    \end{tabular}
  \end{center}

  Let $\alpha$ and $\beta$ satisfy $\supp(\alpha) \subseteq \{1\}$ and
  $\supp(\beta) \subseteq\{2\}$.
  Then, by Theorem~\ref{thm:main_elementary},
  \begin{align*}
    W_{f,\alpha}^\ve & *_Y W_{g,\beta}^\ve
    =
    \frac{1+\alpha(1) X^\ve Y }{(1-Y)(1-X^\ve Y)}
    *_Y
    \frac{1+\beta(2) X^\ve Y }{(1-Y)(1-X^\ve Y)} \\
                     & =
    \frac
    {1 + (1 + \alpha(1) + \beta(2))X^\ve Y +
    (\alpha(1) + \beta(2) + \alpha(1)\beta(2))X^{2\ve}Y
    + \alpha(1)\beta(2) X^{3\ve} Y^2
    }
    {(1-Y)(1-X^\ve Y)(1-X^{2\ve} Y)}
    \\ &
    = W_{f\shuffle g,\,\alpha\cup\beta}^\ve.
  \end{align*}
\end{ex}

\noindent
\begin{cor}
  Let $\ve\in \ZZ$ be fixed.
  Then the set
  \[
    \Bigl\{ W_{f,\alpha}^\ve : (f,\alpha) \text{ is a labelled coloured configuration}\Bigr\}
  \]
  is closed under Hadamard products in $Y$.
\end{cor}
\begin{proof}
  Given coloured configurations $(f,\alpha)$ and $(g,\beta)$, we can find $(g',\beta')$ such that
  $f$ and $g'$ are strongly disjoint and $(g,\beta)\asymp(g',\beta')$.
  In that case, $W_{f,\alpha}^\ve *_Y W_{g,\beta}^\ve = W_{f,\alpha}^\ve *_Y W_{g',\beta'}^\ve =
    W_{f\shuffle g', \,\alpha\beta'}^\ve$ is computed by Theorem~\ref{thm:main_elementary}.
\end{proof}

In Section~\ref{s:appzeta}, we will apply our method based on
Theorem~\ref{thm:main_elementary} to provide explicit formulae for Hadamard
products of ask, class- and orbit-counting zeta functions.

%%%%%%%%%%%%%%%%%%%%%%%%%%%%%%%%%%%%%%%%%%%%%%%%%%%%%%%%%%%%%%%%%%%%%%%%
\section{Coloured quasisymmetric functions}
\label{s:qsym}
%%%%%%%%%%%%%%%%%%%%%%%%%%%%%%%%%%%%%%%%%%%%%%%%%%%%%%%%%%%%%%%%%%%%%%%%

For technical reasons, in this and the following section, we will only
consider coloured permutations with colours drawn from $\{ 0 > 1
  > \dotsb > r-1\} \subset \Gamma$.
For clarity, we occasionally refer to these as \emph{$r$-coloured
  permutations}.
(We similarly refer to $r$-coloured integers, etc.)

%%%%%%%%%%%%%%%%%%%%%%%%%%%%%%%%%%%%%%%%%%%%%%%%
\subsection{Coloured quasisymmetric functions and descent sets}
\label{ss:qsym_sdes}
%%%%%%%%%%%%%%%%%%%%%%%%%%%%%%%%%%%%%%%%%%%%%%%%

\paragraph{Coloured quasisymmetric functions.}
Let $x_i^{(j)}$ for $i=1,2,\dotsc$ and $j=0,1,\dotsc,r-1$ be independent
(commuting) variables. We write $\bx^{(j)} = (x_1^{(j)}, x_2^{(j)}, \dots)$.
The \emph{coloured quasisymmetric function} attached to an $r$-coloured
permutation $\bm a = \sigma^\gamma$ of length $n$ is
\begin{equation}\label{eq:qsym_def_color}
  F_{\bm a} =
  F_{\bm a}(\bx^{(0)}, \dots, \bx^{(r-1)}) =
  \sum_{\substack{1 \le i_1 \le i_2 \le \cdots \le i_n\\j \in \Des^*\!(\bm a) \, \Rightarrow \, i_j < i_{j+1}}} x_{i_1}^{(\gamma_1)}x_{i_2}^{(\gamma_2)}\cdots x_{i_n}^{(\gamma_n)},
\end{equation}
where $\Des^*\!(\bm a) = \Des(\bm a)\setminus \{0\}$.
This is a (homogeneous)
formal power series of degree $n$ in the variables
$\bx^{(0)}, \dots, \bx^{(r-1)}$.
These functions were first introduced in \cite{Poirier}.
For a more comprehensive account, see \cite{BH08}.
Our $F_{\bm a}$ is a special case of $F_{\mathbf{c}}$ from \cite[\S 6.2]{BH08}.
In the literature, authors have employed various choices of total orders in their
definitions of coloured quasisymmetric functions; see, in particular,
\cite{HP/10,Moustakas/th,moustakas21}.
The space $\QSym^{(r)}$ spanned by all coloured quasisymmetric functions
naturally forms a $\QQ$-algebra.

\paragraph{The coloured descent set I.}
It is clear that the coloured quasisymmetric function $F_{\bm a}$ attached to
$\bm a=\sigma^\gamma$ (of length~$n$) only depends on the descent set
$\Des(\bm a)$ and on $\gamma$, rather than on the $r$-coloured permutation
itself.
Both $\Des(\bm a)$ and $\gamma$ can be extracted from the
\emph{coloured descent set} (cf.\ \cite[Definition 2.2]{AAER17}) of $\bm a$
which is defined, for $n > 0$, by
$$\sDes(\bm a)=\bigl\{i^{\gamma_i}: i \in[n-1], \, \gamma_i\neq\gamma_{i+1} \text{ or }
  (\gamma_i=\gamma_{i+1} \text{ and }\sigma_i>\sigma_{i+1})\bigr\} \cup
  \bigl\{n^{\gamma_n}\bigr\};$$
for $n = 0$, we let $\sDes(\bm a) = \varnothing$.
In any case, $\sDes(\bm a)$ is an $r$-coloured subset of $[n]$.

The coloured descent set of $\bm a$ keeps track of (a) the
position and colour of each coloured integer in $\bm a$ that precedes an
entry of a different colour, (b) the position and colour of each
monochromatic descent in $\bm a$, and (c) the position and colour of the
final entry of $\bm a$.
For our recurring example $\bm a = 5^12^23^01^0$, we have
$\sDes(\bm a) = \{ 1^1, 2^2, 3^0, 4^0\}$.
The first two elements (as listed) record colour changes and the third
element records a monochromatic descent.

We write $\mathbb S^{(r)}$ for the set of all $r$-coloured subsets of the set of
positive integers.
We define $\len(\varnothing) = 0$; for a non-empty $A\in \mathbb S^{(r)}$, say
$A = \{ a_1^{\gamma_1}, \dotsc, a_k^{\gamma_k}\}$ with $a_1 < \dotsb < a_k$, we
define $\len(A) = a_k$.
Each element of $\mathbb S^{(r)}$ is of the form $\sDes(\bm a)$ for some
$r$-coloured permutation $\bm a$;
note that $\lvert \bm a\rvert = \len(\sDes(\bm a))$.

Since $F_{\bm a} = F_{\bm b}$ whenever $\sDes(\bm a) = \sDes(\bm b)$,
we also write $F_{A}$ for the coloured quasisymmetric function $F_{\bm a}$
attached to any $r$-coloured permutation $\bm a$ with $\sDes(\bm a)=A$.
Given distinct $A_1,\dotsc,A_\ell \in \mathbb S^{(r)}$, the
coloured quasisymmetric functions $F_{A_1},\dotsc,F_{A_\ell}$ are
linearly independent.

%%%%%%%%%%%%%%%%%%%%%%%%%%%%%%%%%%%%%%%%%%%%%%%%
\subsection{Coloured permutation statistics}
%%%%%%%%%%%%%%%%%%%%%%%%%%%%%%%%%%%%%%%%%%%%%%%%

Let $\sigma = \sigma_1 \dotsb \sigma_n$ and $\pi = \pi_1 \dotsb \pi_n$ be
(uncoloured) permutations of the same length $n$.
That is, $(\sigma_1,\dotsc,\sigma_n)$ and $(\pi_1,\dotsc,\pi_n)$ are sequences
of distinct elements of $\Sigma$.
We say that $\sigma$ and $\pi$ have the \emph{same relative order} if
$\sigma_i \mapsto \pi_i$ defines a poset isomorphism from
$\{\sigma_1,\dotsc,\sigma_n\}$ onto $\{\pi_1,\dotsc,\pi_n\}$.
For example, $132$ and $275$ have the same relative order.

A \emph{coloured permutation statistic} is a function $\st$ defined on the set
of coloured permutations such that given a coloured permutation $\sigma^\gamma$,
if $\pi$ is a permutation of the same length as $\sigma$ and with the same
relative order, then $\st(\sigma^\gamma) = \st(\pi^\gamma)$.
Given coloured permutation statistics $\st_1,\dotsc,\st_k$, we regard the
tuple $(\st_1,\dotsc,\st_k)$ as a coloured permutation statistic via
$(\st_1,\dotsc,\st_k)(\bm a) = (\st_1(\bm a),\dotsc,\st_k(\bm a))$.
Given a coloured permutation $\bm a = \sigma^\gamma = \sigma_1^{\gamma_1}\dotsb \sigma_n^{\gamma_n}$,
let $\col_j(\bm a) := |\{i \in [n] : \gamma_i = j\}|$.
The \emph{colour vector} of an $\bm a$ is
$\bcol(\bm a) = (\col_0(\bm a),\dotsc,\col_{r-1}(\bm a))$;
this is a weak composition of $n$.
The functions $\bcol$, $\des$, and $\comaj$ are coloured permutation statistics.

%%%%%%%%%%%%%%%%%%%%%%%%%%%%%%%%%%%%%%%%%%%%%%%%
\subsection{Coloured shuffle compatibility and shuffle algebras}
\label{ss:shuffles}
%%%%%%%%%%%%%%%%%%%%%%%%%%%%%%%%%%%%%%%%%%%%%%%%

\paragraph{Shuffle compatibility.}
Let $\st$ be a coloured permutation statistic.
Following \cite{GZ18,Moustakas/th}, we say that $\st$ is \emph{shuffle compatible}
if for all disjoint coloured permutations $\bm a$ and $\bm b$, the multiset
$\{\!\{\st(\bm c) : \bm c \in \bm a\shuffle \bm b\}\!\}$
only depends on $\st(\bm a), \st(\bm b)$ and the lengths of $\bm a$ and $\bm b$.
Here, $\bm a\shuffle \bm b$ denotes the set of all coloured permutations
obtained as shuffles of $\bm a$ and $\bm b$.

\paragraph{Shuffle algebras.}
Generalising \cite{GZ18, Moustakas/th},
we associate a \emph{shuffle algebra} $\A_{\st}^{(r)}$ over $\QQ$ to
a shuffle-compatible coloured permutation statistic $\st$ as follows.
First, $\st$ defines an equivalence relation $\sim_{\st}$ on $r$-coloured permutations via
$\bm a\sim_{\st} \bm b$ if and only if
$\bm a$ and $\bm b$ have the same length and
$\st(\bm a)=\st(\bm b)$; we refer to this as
\emph{$\st$-equivalence}.
We write $[\bm a]_{\st}$ to denote the $\st$-equivalence class of
$\bm a$.
As a $\QQ$-vector space $\A_{\st}^{(r)}$ has a basis given by the
$\st$-equivalence classes of $r$-coloured permutations.
The multiplication is given by linearly extending the rule
$$
  [\bm a]_{\st} \, [\bm b]_{\st} = \sum_{\bm c \in \bm a
    \shuffle\bm b} [\bm c]_{\st},
$$where $\bm a$ and $\bm b$ are $r$-coloured permutations on disjoint sets of symbols.
(Thanks to the shuffle compatibility of $\st$, this yields a well-defined
multiplication on $\A_{\st}^{(r)}$.)

Let $\st$ and $\st'$ be shuffle compatible. Suppose that $\st$ \emph{refines}
$\st'$ in the sense that $\st'(\bm a) = \st'(\bm b)$ whenever $\st(\bm a) =
  \st(\bm b)$.
Then the rule $[\bm a]_{\st}\mapsto [\bm a]_{\st'}$ defines a surjective
$\QQ$-algebra homomorphism $\A^{(r)}_{\st} \onto \A^{(r)}_{\st'}$.

\paragraph{The coloured descent set II.}
The following is a restatement of \cite[Eqn (3.4)]{HP/10} (see
also \cite[Thm~4.2.4]{Moustakas/th}). It implies that the coloured descent
set $\sDes$ is shuffle compatible and allows us to identify $\QSym^{(r)}$ and
$\A^{(r)}_{\sDes}$.

\begin{thm}
  Let $\bm a$ and $\bm b$ be symbol-disjoint coloured permutations.
  Let $\sDes(\bm a)=A$, $\sDes(\bm b)=B$, and $n=|\bm a|+|\bm b|$.
  For an $r$-coloured subset $C$ of $[n]$,
  let $\nu_{A,B}^C$ be the number of $r$-coloured permutations ${\bm c}\in
    {\bm a}\shuffle {\bm b}$ such that $\sDes(\bm c)=C$.
  Then
  \begin{equation*}
    F_A F_B = \sum_C \nu_{A,B}^CF_C,
  \end{equation*}
  where $C$ runs over all $r$-coloured subsets of $[n]$. Equivalently,
  \begin{equation*}
     F_{\bm a} F_{\bm b} = \sum_{\bm c \in   {\bm a}\shuffle {\bm b}} F_{\bm c}.
    \end{equation*}
\end{thm}

\begin{cor}[{Cf.\ \cite[Thm 4.4.1]{Moustakas/th}}]
  \label{cor:triple}
  \quad
  \begin{enumerate}
    \item The coloured descent set $\sDes$ is shuffle compatible.
    \item The linear map on $\A^{(r)}_{\sDes}$ defined by $$[\bm a]_{\sDes}
            \mapsto F_{\bm a}$$ is a $\QQ$-algebra isomorphism $\A^{(r)}_{\sDes} \to \QSym^{(r)}$.
          \qed
  \end{enumerate}
\end{cor}

%%%%%%%%%%%%%%%%%%%%%%%%%%%%%%%%%%%%%%%%%%%%%%%%%%%%%%%%%%%%%%%%%%%%%%%%
\section{Descent statistics and coloured shuffle compatibility}
\label{s:coloured}
%%%%%%%%%%%%%%%%%%%%%%%%%%%%%%%%%%%%%%%%%%%%%%%%%%%%%%%%%%%%%%%%%%%%%%%%

As in Section~\ref{s:qsym}, we assume that all coloured permutations, sets, etc.\ are
$r$-coloured for some arbitrary but fixed $r\gg 0$.

%%%%%%%%%%%%%%%%%%%%%%%%%%%%%%%%%%%%%%%%%%%%%%%%
\subsection{Coloured descent statistics}
%%%%%%%%%%%%%%%%%%%%%%%%%%%%%%%%%%%%%%%%%%%%%%%%

As a coloured version of \cite[\S 2.1]{GZ18}, we say
that a coloured permutation statistic $\st$ is a
\emph{coloured descent statistic} if for all coloured permutations $\bm a$ and
$\bm b$, we have $\st(\bm a) = \st(\bm b)$ whenever $\sDes(\bm
  a) = \sDes(\bm b)$.
For $A\in \mathbb S^{(r)}$, we may then unambiguously define
$\st(A) := \st(\bm a)$ where $\bm a$ is any $r$-coloured
permutation with $\sDes(\bm a) = A$.
We write $\st(\mathbb S^{(r)}) = \{ \st(A) : A\in \mathbb
  S^{(r)}\}$.

The following result and its proof constitute a coloured variant of \cite[Thm
4.3]{GZ18} (cf.\ also \cite[Lemma~4.4.4]{Moustakas/th}).
In Theorem~\ref{thm:triple_compatibility},
we will use it to prove that $(\des,\comaj,\bcol)$ is shuffle compatible.

\begin{lem}\label{lem:shuffle_quotient}
  Let $\st$ be a coloured descent statistic.
  Then $\st$ is shuffle compatible if and only
  if there exist
  a $\QQ$-algebra $\mathcal B$
  and a $\QQ$-algebra homomorphism $\phi\colon  \QSym^{(r)} \to
    \mathcal B$ with the following properties.
  \begin{enumerate}[(i)]
    \item \label{lem:shuffle_quotient1}
          For all $A,B\in \mathbb S^{(r)}$ with $\st(A) = \st(B)$ and $\len(A) =
            \len(B)$ (see Section~\ref{ss:qsym_sdes}), we have $\phi(F_A) = \phi(F_B)$.
    \item \label{lem:shuffle_quotient2}
          For each (finite or infinite) sequence $A_1,A_2,\dotsc \in \mathbb S^{(r)}$
          such that the $(\st(A_i),\len(A_i))$ are pairwise distinct, the images
          $\phi(F_{A_1}), \phi(F_{A_2}),\dotsc$ are $\QQ$-linearly independent.
  \end{enumerate}
  In this case, the rule $[\bm a]_{\st}\mapsto
    \phi(F_{\bm a})$ yields an injective $\QQ$-algebra homomorphism
  $\mathcal A^{(r)}_{\st}\to \mathcal B$.
\end{lem}
\begin{proof}
  If $\st$ is shuffle compatible, then
  Corollary~\ref{cor:triple} yields a homomorphism
  \begin{equation}
    \label{eq:phi_st}
    \phi_{\st}\colon \QSym^{(r)} \xrightarrow{\cong} \mathcal A^{(r)}_{\sDes}
    \onto \mathcal A^{(r)}_{\st}
  \end{equation}%
  such that $\phi_{\st}(F_{\bm a})= [\bm a]_{\st}$ for
  each coloured permutation $\bm a$.
  By construction, $\phi_{\st}$ has the desired properties.
  Conversely, suppose that $\phi$ satisfies \ref{lem:shuffle_quotient1}--\ref{lem:shuffle_quotient2}.
  Let $\bm a$ and $\bm b$ be (symbol-)disjoint coloured permutations.
  Using Corollary~\ref{cor:triple}, we obtain
  \begin{equation}
    \label{eq:recover}
    \phi(F_{\bm a})\phi(F_{\bm b}) = \phi(F_{\bm a} F_{\bm b})
    = \phi\left(\sum_{\bm c \in \bm a \shuffle \bm b} F_{\bm c}\right) = \sum_{\bm c \in
      \bm a\shuffle \bm b} \phi(F_{\bm c}).
  \end{equation}%
  Conditions \ref{lem:shuffle_quotient1}--\ref{lem:shuffle_quotient2} allow us to
  recover the multiset $\{\!\{\st(\bm c) : \bm c \in \bm a\shuffle \bm
    b\}\!\}$ from the quadruple $(\st(\bm a), \st(\bm b), \lvert \bm a\rvert,
    \lvert \bm b\rvert)$.
  Indeed, the multiplicity of $\st(\bm c)$ in our multiset is the coefficient
  of $\phi(F_{\bm c})$ in \eqref{eq:recover}.
  We conclude that $\st$ is shuffle compatible.

  Finally, let $\st$ be shuffle compatible and let $\phi$ with
  \ref{lem:shuffle_quotient1}--\ref{lem:shuffle_quotient2} be given.
  Define $\phi_{\st}$ as in~\eqref{eq:phi_st}.
  Then $\phi_{\st}(F_{\bm a}) = \phi_{\st}(F_{\bm b})$ if and only if $\lvert\bm
    a\rvert = \lvert\bm b\rvert$ and $\st(\bm a) = \st(\bm b)$
  if and only if $\phi(F_{\bm a}) = \phi(F_{\bm b})$.
  The kernel of $\phi_{\st}$ is generated by all $F_{\bm a}-F_{\bm b}$ with
  $\lvert\bm a\rvert = \lvert\bm b\rvert$ and $\st(\bm a) = \st(\bm b)$.
  In particular, $\Ker(\phi_{\st})\subseteq \Ker(\phi)$ and
  the rule $[\bm a]_{\st} \mapsto \phi(F_{\bm a})$ indeed defines a
  $\QQ$-algebra homomorphism $\psi\colon \A^{(r)}_{\st} \to \mathcal B$ with $\phi =
    \psi \circ \phi_{\st}$.
  By~\ref{lem:shuffle_quotient2}, $\psi$ is injective.
\end{proof}

%%%%%%%%%%%%%%%%%%%%%%%%%%%%%%%%%%%%%%%%%%%%%%%%
\subsection{The shuffle algebra of $(\des,\comaj,\bcol)$}
\label{ss:A_des_comaj_p}
%%%%%%%%%%%%%%%%%%%%%%%%%%%%%%%%%%%%%%%%%%%%%%%%

The main shuffle algebra of interest to us
is the one attached to $(\des,\comaj,\bcol)$. The latter statistic is
naturally refined by the coloured descent set $\sDes$.
Let $p_0,\dotsc,p_{r-1},x,t$ be commuting variables over $\QQ$.
Write $\bp = (p_0,\dotsc,p_{r-1})$ and $\bp^{\bm v} = p_0^{v_0}\dotsb
p_{r-1}^{v_{r-1}}$.
For a ring $R$, let $R\llbracket t\Asterisk\rrbracket$ denote the ring
$R\llbracket t\rrbracket$ with multiplication given by the Hadamard product in~$t$.
Equivalently, by identifying a formal power series and its sequence of
coefficients, we may identify $R\llbracket t\Asterisk\rrbracket$ and the
product ring $R^{\mathbb N_0}$.

\begin{thm}\label{thm:triple_compatibility}
  \quad
  \begin{enumerate}[(i)]
    \item\label{thm:triple_compatibility1}
          The tuple of statistics $(\des,\comaj,\bcol)$ is shuffle compatible.
    \item\label{thm:triple_compatibility2}
          The linear map $H\colon \A_{(\des,\comaj,\bcol)}^{(r)} \to \QQ[\bm p, x]
            \llbracket t\Asterisk\rrbracket$ defined by
          \begin{equation}
            \label{eq:triple_compatibility}
            [\bm a]_{(\des,\comaj,\bcol)} \mapsto
            \frac{\bp^{\bcol(\bm a)} x^{\comaj(\bm a)} t^{\des(\bm a)}
            }{(1-t)(1-xt)\cdots(1-x^{\lvert \bm a\rvert}t)}
          \end{equation}
          is an injective algebra homomorphism.
  \end{enumerate}
\end{thm}

Our proof of Theorem~\ref{thm:triple_compatibility} will be based on the
following strategy.
First, we construct a sequence of judicious specialisation homomorphisms $\psi_m$
defined on $\QSym^{(r)}$.
Next, these $\psi_m$ can be combined to form a single homomorphism
$\QSym^{(r)} \to \QQ[\bm p, x]\llbracket t\Asterisk\rrbracket$ which sends
each $F_{\bm a}$ to the right-hand side of \eqref{eq:triple_compatibility}.
Both parts of Theorem~\ref{thm:triple_compatibility} then follow by
invoking Lemma~\ref{lem:shuffle_quotient}.
This strategy and its execution are inspired by \cite[Thm~4.5]{GZ18} and
\cite[Thm~4.4.3]{Moustakas/th}.
In particular, our Proposition~\ref{prop:qsym_specialization_color}
is analogous to \cite[Eqn~(2.12)]{Moustakas/th}.
(For related but coarser results, see \cite{moustakas21,moustakas21b}.)
Our Theorem~\ref{thm:triple_compatibility} can be seen as a refined version of
\cite[Thm~4.4.3]{Moustakas/th}.
We note that using a suitable coloured version of \cite[Lemma~3.6]{GZ18}, we
could exchange the comajor indices in the present article for the major
indices used in \cite{Moustakas/th}, thus obtaining results that are more
direct refinements of those from \cite{Moustakas/th}.

Let $\psi_m\colon \QSym^{(r)} \to \QQ[{\bm p},x]$ be the specialisation
defined by the substitution
\[
  \begin{cases}
    x_i^{(0)} \gets x^{i-1}p_0, & \quad \text{if $1 \le i \le m$},                     \\
    x_i^{(j)} \gets x^{i-1}p_j, & \quad \text{if $1 < i \le m$ and $1 \le j \le r-1$}, \\
    x_i^{(j)} \gets 0,          & \quad \text{otherwise,}
  \end{cases}
\]
where the notation ``$y \gets z$'' indicates that $y$ is to be replaced with $z$.
This is well defined since all but finitely many variables are sent to zero.
Note that each $\psi_m$ is an algebra homomorphism.
Further note that $\psi_m(x_i^{(j)}) = 0$ if and only if $i > m$ or if $i = 1$
and $j > 0$.
Define a $\QQ$-linear map $\Psi\colon \QSym^{(r)}\to \QQ[\bp, x]\llbracket t\Asterisk\rrbracket$ via
$$
  \Psi(F_{\bm a}) = \sum_{m=1}^\infty \psi_m(F_{\bm a})t^{m-1}.
$$%
Clearly, $\Psi$ is an algebra homomorphism.

\begin{pro}
  \label{prop:qsym_specialization_color}
  For each $r$-coloured permutation $\bm a$ with $\lvert \bm a\rvert =n$, we have
  \begin{align}
    \Psi(F_{\bm a})
     & =  \frac{ \bp^{\bcol(\bm a)}x^{\comaj(\bm a)}t^{\des(\bm a)}}{(1-t)(1-xt)\cdots(1-x^nt)}.  \label{eq:descomaj_color}
  \end{align}
\end{pro}
\begin{proof}
  Let $\bm a=\sigma^\gamma$ be an $r$-coloured permutation of length $n$.
  When applying $\psi_m$ to $F_{\bm a}$, a summand
  $s := x_{i_1}^{(\gamma_1)}x_{i_2}^{(\gamma_2)}\cdots x_{i_n}^{(\gamma_n)}$
  is sent to zero if and only if $i_j > m$ for some $j$ or both $i_1 = 1$ and
  $\gamma_1 \neq 0$.
  Writing $i_0 = 1$ from now on, we can equivalently express the latter
  condition as $0\in \Des(\bm a)$ and $i_0 \nless i_1$.
  In all other cases, $s$ is sent to $\bp^{\bcol(\bm a)} x^{i_1+\dotsb + i_n - n}$.
  Using our convention $i_0 = 1$,
  we conclude that
  \begin{equation}
    \psi_m \left(F_{\bm a}\right) = \sum_{\substack{1 = i_0 \le i_1 \le i_2 \le \cdots \le
        i_n \le m \\j \in \Des(\bm a) \, \Rightarrow \, i_j < i_{j+1}}}
    \bp^{\bcol(\bm a)}
    x^{i_1 + i_2 + \cdots + i_n - n}. \label{eq:Heelp}
  \end{equation}
  In order to eliminate the strict inequalities on the right-hand side of
  \eqref{eq:Heelp}, we write
  \begin{align*}
    i_n'     & = i_n - \chi_0 - \chi_1 - \cdots - \chi_{n-1},     \\
    i_{n-1}' & = i_{n-1} - \chi_0 - \chi_1 - \cdots - \chi_{n-2}, \\
             & \phantom=\!\vdots                                  \\
    i_2'     & = i_2 - \chi_0 - \chi_1,                           \\
    i_1'     & = i_1 - \chi_0,
  \end{align*}
  where
  \[
    \chi_i :=
    \begin{cases}
      1, & \quad \text{if $i \in \Des(\bm a)$}, \\
      0, & \quad \text{otherwise}.
    \end{cases}
  \]%

The inequalities under the summation sign in \eqref{eq:Heelp} can be rewritten as $1\le i_1'\le \cdots \le i_n'\le m-\des(\bm a)$.
Next,
\[
  i_1 + i_2 + \cdots + i_n = i_1' + i_2' + \cdots + i_n' +
  \sum_{i=0}^{n-1} (n-i)\chi_i
\]
and
\[
  \sum_{i=0}^{n-1} (n-i)\chi_i
  = \sum_{i \in \Des(\bm a)} (n-i) = \comaj(\bm a).
\]
Equation~\eqref{eq:Heelp} therefore becomes
\begin{equation}
  \psi_m \left(F_{\bm a}\right) = \sum_{1 \le i_1' \le i_2' \le \cdots \le
    i_n' \le m-\des(\bm a)}
  \bp^{\bcol(\bm a)}
  x^{i_1' + i_2' + \cdots + i_n' - n + \comaj(\bm a)}.
  \label{eq:Heeelp}
\end{equation}
Writing $m' = m - \des(\bm a)$, using \eqref{eq:Heeelp}, we obtain
\begin{align*}
  \Psi(F_{\bm a}) & = \sum_{m=1}^\infty \psi_m(F_{\bm a}) t^{m-1} \\
                  & = \bp^{\bcol(\bm a)}x^{\comaj(\bm a)-n} \, t^{\des(\bm a)-1}
                    \sum_{1 \le i_1' \le \dotsb \le i'_{n} \le m'}
                    x^{i_1'+\dotsb+i_n'} \, t^{m'}
\end{align*}
whence the claim follows from the identity of formal power series
\[
  \sum_{1\le j_1\le\dotsb \le j_{n+1}} \lambda_1^{j_1}\dotsb \lambda_{n+1}^{j_{n+1}}
  = \frac{\lambda_1\dotsb \lambda_{n+1}}
  {
    (1- \lambda_1\dotsb \lambda_{n+1})
    \dotsb
    (1 - \lambda_{n}\lambda_{n+1})
    (1-\lambda_{n+1})
  }
\]
via the substitution $(\lambda_1,\dotsc,\lambda_n,\lambda_{n+1}) \gets (x,\dotsc,x,t)$.
\end{proof}

\begin{lem}
  \label{lem:substantial_leadership}
  Let $\bm a_1,\dotsc,\bm a_k$ be $r$-coloured permutations such that the triples
  $$(\des(\bm a_i), \comaj(\bm a_i), \bcol(\bm a_i))$$ are pairwise distinct.
  Then $\Psi(F_{\bm a_1}),\dotsc,\Psi(F_{\bm a_k})$ are linearly independent over $\QQ$.
\end{lem}
\begin{proof}
  First note that for an $r$-coloured permutation $\bm b$, we have
  \begin{equation}
    \label{eq:substantial_leadership}
    \Psi(F_{\bm b}) =
    \bp^{\bcol(\bm b)} x^{\comaj(\bm b)} t^{\des(\bm b)} + \mathcal
    O(t^{\des(\bm b)+1}).
  \end{equation}%
  Suppose, for the sake of contradiction, that the claim is false.
  Choose $k$ minimal such that $\Psi(F_{\bm a_1}),\dotsc,\Psi(F_{\bm a_k})$ are
  linearly dependent, say $\sum_{i=1}^k \lambda_i \Psi(F_{\bm a_i}) = 0$ for
  nonzero $\lambda_i\in \QQ$.
  Let $d = \min(\des(\bm a_1),\dotsc,\des(\bm a_k))$.
  After suitably permuting indices, we may assume that for some $1\le \ell\le k$,
  we have $\des(\bm a_i) = d$ if and only if $i\le \ell$.
  Then \eqref{eq:substantial_leadership} shows that
  $$
    0 \equiv \sum_{i=1}^k \lambda_i \Psi(F_{\bm a_i})
    \equiv \sum_{i=1}^\ell \lambda_i \Psi(F_{\bm a_i})
    \equiv \sum_{i=1}^\ell \lambda_i
    \bp^{\bcol(\bm a_i)} x^{\comaj(\bm a_i)} t^d
    \pmod{t^{d+1}}
  $$
  and thus $\sum_{i=1}^\ell \lambda_i \bp^{\bcol(\bm a_i)} x^{\comaj(\bm a_i)}
    = 0$.
  By assumption, the $(\bcol(\bm a_i), \comaj(\bm a_i))$ are pairwise distinct for
  $i = 1,\dotsc,\ell$.
  As distinct monomials are linearly independent, we conclude that $\lambda_1
    = \dotsb = \lambda_\ell = 0$, contradicting the minimality of our counterexample.
\end{proof}

\begin{proof}[Proof of Theorem~\ref{thm:triple_compatibility}]
  Equation~\eqref{eq:descomaj_color} and
  Lemma~\ref{lem:substantial_leadership} show that $\Psi$ satisfies the
  conditions on $\phi$ in
  Lemma~\ref{lem:shuffle_quotient}\ref{lem:shuffle_quotient1} and
  \ref{lem:shuffle_quotient2}, respectively.
  This concludes the proof.
\end{proof}

We record the following alternative form of
Theorem~\ref{thm:triple_compatibility}\ref{thm:triple_compatibility2} which
more closely resembles its ancestor \cite[Thm~4.15(c)]{GZ18} and relative
\cite[Thm~4.4.3]{Moustakas/th}.
The map $\tilde H$ in Corollary~\ref{cor:des+1} only differs from $H$ in
Theorem~\ref{thm:triple_compatibility} in the factor $z^{\abs{\bm
        a}}$ and an additional factor $t$ for $\abs{\bm a} > 0$.

\begin{cor}
  \label{cor:des+1}
  The linear map $\tilde H\colon \A_{(\des,\comaj,\bcol)}^{(r)} \to \QQ[\bm p, x, z]
    \llbracket t\Asterisk\rrbracket$ defined by
  $$
    [\bm a]_{(\des,\comaj,\bcol)} \mapsto
    \begin{cases}
      \frac{\bp^{\bcol(\bm a)} x^{\comaj(\bm a)} t^{\des(\bm a)+1}
      }{(1-t)(1-xt)\cdots(1-x^{\lvert \bm a\rvert}t)} z^{\lvert \bm a\rvert},
                     & \text{if $\lvert \bm a\rvert \ge 1$} \\
      \frac{1}{1-t}, & \text{if $\lvert \bm a\rvert=0$}
    \end{cases}
  $$ is an injective algebra homomorphism.
\end{cor}
\begin{proof}
  We write $[\cdot]$ instead of  $[\cdot]_{(\des,\comaj,\bcol)}$.
  Consider the injective homomorphism $\mu\colon \QQ[\bp, x]\to \QQ[\bp,x,z]$ which
  fixes $x$ and sends each $p_i$ to $p_i z$. Let $M$ be the injective
  homomorphism $\QQ[\bp, x]\llbracket t\Asterisk\rrbracket \to
    \QQ[\bp,x,z]\llbracket t\Asterisk\rrbracket$ with $M(\sum_{m=0}^\infty a_m t^m) =
    \sum_{m=0}^\infty \mu(a_m)t^m$.
  Let $H$ be as in Theorem~\ref{thm:triple_compatibility}.
  Then $J := M\circ H$ is the injective homomorphism which sends
  each $[\bm a]$ to
  $$\frac{\bp^{\bcol(\bm a)} x^{\comaj(\bm a)} t^{\des(\bm a)}
    }{(1-t)(1-xt)\cdots(1-x^{\lvert \bm a\rvert}t)} z^{\lvert \bm a\rvert}.$$
  By construction, $\tilde H([\bm a]) = t\cdot J([\bm a])$ for $\abs{\bm a} > 0$.
  Then $\tilde H$ and $t \cdot J$ agree on the $\QQ$-span, $L$ say, of all
  $[\bm a]$ with $\abs{\bm a}> 0$.
  We thus conclude that for $\abs{\bm a}, \abs{\bm b} > 0$, we have
  \[
    \tilde H([\bm a][\bm b])  = t \cdot J([\bm a][\bm b])
    = t\cdot (J([\bm a]) *_t J([\bm b]))
    = (t \cdot J([\bm a])) *_t (t \cdot J([\bm b])) \\
    = \tilde H([\bm a]) *_t \tilde H([\bm b]).
  \]
  Next, $\tilde H$ is injective on $L$ since $\tilde H$ and $t \cdot J$ agree
  on $L$.
  Finally, for the unique coloured permutation $\bm e$ with $\lvert \bm e\rvert =
  0$, we have $\tilde H([\bm e]) \notin \tilde H(L)$.
  Indeed, the constant term (as a series in~$t$) of each element of
  $\tilde H(L)$ is zero. 
  Hence, $\tilde H$ is injective on $\A_{(\des,\comaj,\bcol)}^{(r)}$.
\end{proof}

%%%%%%%%%%%%%%%%%%%%%%%%%%%%%%%%%%%%%%%%%%%%%%%%
\subsection{Proof of Theorem~\ref{thm:main_elementary}}
\label{ss:proof_elementary}
%%%%%%%%%%%%%%%%%%%%%%%%%%%%%%%%%%%%%%%%%%%%%%%%

We again write $[\cdot]$ instead of $[\cdot]_{(\des,\comaj,\bcol)}$.
Let $(f,\alpha)$ and $(g,\beta)$ be coherent labelled coloured
configurations as in Section~\ref{elementary}.
Choose $r \gg 0$ such that the coloured permutations appearing in $f$ and
$g$ only involve colours from $\{0,\dotsc,r-1\}$.
Let $H$ be as in Theorem~\ref{thm:triple_compatibility}.
Define $\rho_0,\dotsc,\rho_{r-1}\in \QQ(X)$ via
$$
  \rho_i = \begin{cases}
    \alpha(i), & \text{if } i\in \palette^*(f), \\
    \beta(i),  & \text{if } i\in \palette^*(g), \\
    1,         & \text{otherwise.}
  \end{cases}
$$%
The $\rho_i$ are well defined thanks to the coherence of $(f,\alpha)$ and
$(g,\beta)$; note that $\rho_0 = 1$.
Writing $\bm\rho = (\rho_0,\dotsc,\rho_{r-1})$,
we then find that $\bm\rho^{\bcol(\bm a)} = (\alpha\cup\beta)(\bm a)$
for each $r$-coloured permutation $\bm a$. \phantom{\qedhere}

Fix $\varepsilon \in \ZZ$.
Let $K\colon \QQ[\bm p,x]\llbracket t\Asterisk\rrbracket \to \QQ(X)\llbracket Y\!\Asterisk\rrbracket$ be
the algebra homomorphism which sends each $\sum_{n=0}^\infty
  h_n(p_0,\dotsc,p_{r-1},x)t^n$ to
$\sum_{n=0}^\infty h_n(\rho_0,\dotsc,\rho_{r-1},X^\varepsilon) Y^n$.
Let $L = K \circ H \colon \A_{(\des,\comaj,\bcol)}^{(r)} \to
  \QQ(X)\llbracket Y\!\Asterisk\rrbracket$.
Write $F = \sum_{\bm a\in \supp(f)} f_{\bm a} [\bm a]$ and $G = \sum_{\bm b \in
    \supp(g)} g_{\bm b} [\bm b]$.
By construction, we then have $L(F) = W^\varepsilon_{f,\alpha}$, $L(G) =
  W^\varepsilon_{g,\beta}$, and $L(FG) = W^{\varepsilon}_{f\shuffle g,
      \alpha\cup\beta}$.
The claim follows since
$$
  W^{\varepsilon}_{f\shuffle g, \alpha\cup\beta}
  = L(FG) = L(F)*_Y L(G) = W^\varepsilon_{f,\alpha} *_Y  W^\varepsilon_{g,\beta}.\eqno\qed
$$

%%%%%%%%%%%%%%%%%%%%%%%%%%%%%%%%%%%%%%%%%%%%%%%%%%%%%%%%%%%%%%%%%%%%%%%%
\section{Applications to zeta functions}
\label{s:appzeta}
%%%%%%%%%%%%%%%%%%%%%%%%%%%%%%%%%%%%%%%%%%%%%%%%%%%%%%%%%%%%%%%%%%%%%%%%

\subsection{Ask, class-counting, and orbit-counting zeta functions}
\label{ss:ask}

The main purpose of the present section is to recall several explicit families
of zeta functions associated with algebraic structures (Examples~\ref{ex:ask_mat}--\ref{ex:graphs}).
These families will feature in our applications of
Theorem~\ref{thm:main_elementary} in Section~\ref{ss:app}.
For further details on and motivation for the study of these zeta functions,
see \cite{Rossmann/18,ask2}.
In order to maintain consistency with the literature, we
regard $d\times e$ matrices over a ring $R$ as homomorphisms $R^d \to
  R^e$ acting by right multiplication.

\paragraph{Global ask zeta functions.}
Given a module $M\subseteq \Mat_{d\times e}(\ZZ)$ of integral matrices, for
each $n\ge 1$, let $M_n\subseteq \Mat_{d \times e}(\ZZ/n\ZZ)$ denote the
reduction of $M$ modulo $n$.
The \emph{(global) ask zeta function} \cite[Defn~3.1(i)]{Rossmann/18} of $M$
is the Dirichlet series $\zeta_M^{\ask}(s) = \sum_{n=1}^\infty a_n(M) n^{-s}$,
where $a_n(M) \in \QQ$ denotes the \underline average \underline size of the
\underline kernel of matrices in $M_n$.
By the Chinese remainder theorem,
$\zeta^{\ask}_M(s) = \prod_p \zeta^{\ask}_{M_p}(s)$ (Euler product), where
the product is taken over all primes $p$ and the \emph{local factor} at $p$ is
given by $\zeta^{\ask}_{M_p}(s) = \sum_{k=0}^\infty a_{p^k}(M) p^{-ks}$,
a power series in $p^{-s}$;
see \cite[Prop.\ 3.4(ii)]{Rossmann/18}.
Drawing upon deep results from $p$-adic integration and the theory of zeta
functions of algebraic structures,
it is known that each $\zeta^{\ask}_{M_p}(s)$ is rational in $p^{-s}$;
see \cite[Thm~1.4]{Rossmann/18}.
Moreover, as $p$ varies, $\zeta^{\ask}_{M_p}(s)$ is expressible as a weighted
sum of rational functions in $p$ and~$p^{-s}$, where the weights
are the numbers of $\FF_p$-rational points on certain $\ZZ$-defined varieties;
see \cite[Thm~4.11]{Rossmann/18}.
Thus, these zeta functions generally exhibit delicate arithmetic features.
In the present article, we will exclusively focus on so-called uniform
examples.
These are given by modules of matrices $M$ for which there exists a single
bivariate rational function $W(X,Y)\in \QQ(X,Y)$ such that
$\zeta^{\ask}_{M_p}(s) = W(p,p^{-s})$ for all primes $p$, perhaps allowing for
finitely many exceptions.

\paragraph{Local ask zeta functions.}
It is often advantageous to bypass global structures altogether and directly
study variants of the local factors from above.
Let $\mathfrak O$ be a compact discrete valuation ring.
Let $\mathfrak P$ be the maximal ideal of $\mathfrak O$ and let $q$
denote the size of the residue field $\mathfrak O/\mathfrak P$.
Such rings $\mathfrak O$ are precisely the valuation rings of non-Archimedean
local fields.
Examples include the ring of $p$-adic integers $\ZZ_p$ (in which case
$\mathfrak P = p\ZZ_p$, $q = p$, and
$\mathfrak O/\mathfrak P\cong \FF_p$) and the ring $\FF_q\llbracket z\rrbracket$ of formal
power series over $\FF_q$ (in which case $\mathfrak P = z \FF_q\llbracket z\rrbracket$).

Given a module of matrices $M\subseteq \Mat_{d\times e}(\mathfrak O)$,
its associated \emph{(local) ask zeta function} is the formal power series
$\Zeta^{\ask}_M(Y) = \sum_{k=0}^\infty \alpha_k(M) Y^k$, where $\alpha_k(M)$
denotes the average size of the kernels within the reduction of $M$ modulo
$\mathfrak P^k$.
Local ask zeta functions generalise local factors of global ask zeta
functions.
Indeed, for a submodule $M\subseteq \Mat_{d\times e}(\ZZ)$ and prime $p$,
let~$M_p$ denote the $\ZZ_p$-submodule of $\Mat_{d\times e}(\ZZ_p)$ generated
by $M$.
Then $\Zeta^{\ask}_{M_p}(p^{-s})$ coincides with the local factor
$\zeta^{\ask}_{M_p}(s)$ as defined above.

Explicit formulae for ask zeta functions associated with various ``well-known''
families of modules of matrices are recorded in \cite[\S
  5]{Rossmann/18}.
The following examples are of particular interest to us here.

\begin{ex}
  \label{ex:ask_mat}
  $\Zeta^{\ask}_{\Mat_{d\times e}(\mathfrak O)}(Y) = \frac{1-
    q^{-e}Y}{(1-Y)(1-q^{d-e}Y)}$; see \cite[Prop.~1.5]{Rossmann/18}.
\end{ex}

\begin{ex}
  \label{ex:ask_so}
  Let $\mathfrak O$ have characteristic distinct from $2$.
  Let $\mathfrak{so}_d(\mathfrak O)$ be the module of antisymmetric $d\times d$
  matrices over $\mathfrak O$.
  By \cite[Prop.~5.11]{Rossmann/18},
  $\Zeta^{\ask}_{\mathfrak{so}_d(\mathfrak O)}(Y) = \frac{1-q^{1-d}
    Y}{(1-Y)(1-qY)}
  =\Zeta^{\ask}_{\Mat_{d\times (d-1)}(\mathfrak O)}(Y)
  $.
\end{ex}

\begin{ex}
  \label{ex:ask_n}
  Let $\mathfrak n_d(\mathfrak O)$ be the module of strictly upper triangular
  $d\times d$ matrices over $\mathfrak O$.
  By \cite[Prop.~5.15(i)]{Rossmann/18},
  $\Zeta^{\ask}_{\mathfrak n_d(\mathfrak O)}(Y) = \frac{(1-Y)^{d-1}}{(1-qY)^d}$.
\end{ex}

\paragraph{Class- and orbit-counting zeta functions.}
Let $\mathfrak O$ be a compact discrete valuation ring as above.
Let $\mathsf G$ be a linear group scheme over $\mathfrak O$, with a given embedding
into $d\times d$ matrices.

The \emph{orbit-counting zeta function} of $\mathsf G$ is the generating function
$\Zeta^{\oc}_{\mathsf G}(Y) = \sum_{k=0}^\infty b_k(\mathsf G) Y^k$,
where $b_k(\mathsf G)$ denotes the number of orbits of the (finite) matrix
group $\mathsf G(\mathfrak O/\mathfrak P^k)$ acting by right-multiplication on
its natural module $(\mathfrak O/\mathfrak P^k)^d$.
Apart from the given linear action of $\mathsf G$, it is also natural to let
$\mathsf G$ act on itself by conjugation.
The \emph{class-counting zeta function} of $\mathsf G$ is the generating
function $\Zeta^{\cc}_{\mathsf G}(Y) = \sum_{k=0}^\infty c_k(\mathsf G) Y^k$,
where $c_k(\mathsf G)$ denotes
the number of conjugacy classes of $\mathsf G(\mathfrak O/\mathfrak P^k)$.
Class-counting zeta functions go back to work of
du~Sautoy~\cite{dS05}.
As shown in \cite[\S 8.5]{Rossmann/18} and \cite[\S 6.2]{ask2}, if $\mathsf G$
is unipotent, then, subject to (mild) restrictions on the residue field size
$q$ of $\mathfrak O$, the class- and orbit-counting zeta functions of $\mathsf
G$ are instances of ask zeta functions associated with modules of matrices
over $\mathfrak O$.
(These modules can be explicitly described in terms of the Lie algebra
of~$\mathsf G$.)
When passing between ask and class-counting zeta functions, one
often needs to apply a transformation $Y\gets q^m Y$ for a suitable integer
$m$; see below for an example.

\begin{ex}
  \label{ex:cc_F2d}
  Suppose that the residue field size $q$ of $\mathfrak O$ is odd.
  By exponentiation, the free class-$2$-nilpotent Lie algebra on $d$
  generators over $\mathfrak O$ gives rise to a group scheme $\mathsf F_{2,d}$ over
  $\mathfrak O$.
  We may view $\mathsf F_{2,d}$ as an analogue of the free
  class-$2$-nilpotent group on $d$ generators.
  Lins \cite[Cor.\ 1.5]{Lins/20} showed that
  $$\Zeta^{\cc}_{\mathsf F_{2,d}}(Y) = \frac{1-q^{\binom{d-1} 2}Y}{\Bigl(1 - q^{\binom d
      2}Y\Bigr)\Bigl(1- q^{\binom d 2 + 1}Y\Bigr)}.$$%
  Looking back at Example~\ref{ex:ask_so}, we observe that
  $\Zeta^{\cc}_{\mathsf F_{2,d}}(Y) = \Zeta^{\ask}_{\mathfrak{so}_d(\mathfrak
    O)}(q^{\binom d 2}Y)$;
  this is no coincidence, see \cite[Ex.~7.3]{ask2}.
\end{ex}

\begin{ex}
  \label{ex:oc_Ud}
  Let $\mathsf U_d$ be the group scheme of upper unitriangular $d\times d$ matrices over
  $\mathfrak O$.
  Suppose that $\gcd(q, (d-1)!) = 1$.
  By \cite[Thm~1.7]{Rossmann/18} (cf.~\cite[Prop.\ 4.12]{board})
  and Example~\ref{ex:ask_n}, we have
  $\Zeta^{\oc}_{\mathsf U_d}(Y) = \Zeta^{\ask}_{\mathfrak n_d(\mathfrak O)}(Y) = \frac{(1-Y)^{d-1}}{(1-qY)^{d}}$.
\end{ex}

\paragraph{Graphs and graphical groups.}
Given a (finite, simple) graph $\mathsf \Gamma$ with distinct vertices
$v_1,\dotsc,v_n$.
let $M_{\mathsf\Gamma}$ be the $\mathfrak O$-module of alternating
$n\times n$ matrices $A = [a_{ij}]$ such that $a_{ij} = 0$ whenever $v_i$ and $v_j$ are
non-adjacent.
(Here, a matrix is alternating if it is antisymmetric and all diagonal entries
are zero. Alternating and antisymmetric matrices coincide over rings in which
$2$ is not a zero divisor. In the case of $\mathfrak O$, this amounts to $0 \not= 2$.)
We write $\Zeta^{\ask}_{\mathsf\Gamma}(Y)$ for $\Zeta^{\ask}_{M_{\mathsf\Gamma}}(Y)$.
As shown in \cite[Thm~A]{cico}, $\Zeta^{\ask}_{\mathsf\Gamma}(Y)$ is a rational
function in $q$ and $Y$, without any restrictions on $q$.
In \cite[\S 3.4]{cico}, the graphical group scheme $\mathsf G_{\mathsf\Gamma}$
associated with $\mathsf\Gamma$ is constructed; for an alternative but equivalent
construction, see \cite[\S 1.1]{higman}.
By \cite[Prop.~3.9]{cico}, $\Zeta^{\cc}_{\mathsf G_{\mathsf\Gamma}}(Y) =
  \Zeta^{\ask}_{\mathsf\Gamma}(q^m Y)$, where $m$ is the number of edges of $\mathsf\Gamma$.
Given graphs $\mathsf\Gamma_1$ and $\mathsf\Gamma_2$, let $\mathsf\Gamma_1 \vee \mathsf\Gamma_2$ denote
their join, obtained from the disjoint union
$\mathsf\Gamma_1\oplus\mathsf\Gamma_2$ of $\mathsf\Gamma_1$ and $\mathsf\Gamma_2$ by
adding edges connecting each vertex of $\mathsf\Gamma_1$ to each vertex of $\mathsf\Gamma_2$.
Let $\mathsf K_n$ (resp.\ $\mathsf\Delta_n$) denote the complete (resp.~edgeless)
graph on $n$ vertices.

\begin{ex}
  \label{ex:graphs}
  The graph $\mathsf\Delta_n \vee \mathsf K_{n+1}$ is the threshold graph
  $\operatorname{Thr}(n,n+1)$
  in the notation from \cite[\S 8.4]{cico};
  this graph has $3\binom{n+1} 2$ edges.
  It follows from \cite[Thm~8.18]{cico} that
  $$
    \Zeta^{\ask}_{\mathsf\Delta_n \vee \mathsf K_{n+1}}(Y) = \frac{(1-q^{-n}Y)(1-q^{-n-1}Y)}{(1-q^{-1}Y)(1-Y)(1-qY)}.
  $$
\end{ex}

\begin{ex}
  \label{ex:more_graphs}
  In the same spirit,
  $\mathsf T_n := ((\mathsf\Delta_n \vee \mathsf K_{n+1}) \oplus \mathsf\Delta_{n+2}) \vee \mathsf K_{n+4} $ is the threshold graph
  $\operatorname{Thr}(n,n+1,n+2,n+4)$.
  This graph has $3n^2 + 12n + 18 = \binom{n+4}2 + 3\binom{n+1}2 + (n+4)(3n+3)$ edges.
  By \cite[Thm~8.18]{cico},
  $$
    \Zeta^{\ask}_{\mathsf T_n}(Y) = \frac{ (1-q^{-n-4}Y)(1-q^{-n-3}Y)^2 (1-q^{-n-2}Y)}{(1-q^{-3}Y)(1-q^{-2}Y)(1-q^{-1}Y)(1-Y)(1-qY)}.
  $$
\end{ex}

\paragraph{Hadamard products and zeta functions.}

Let $\mathfrak O$ be as above.
Elaborating further on what we wrote in the introduction,
modules of matrices, (linear) group schemes, and graphs all admit natural
operations which correspond to taking Hadamard products of associated zeta
functions.
In detail, given modules $M \subseteq \Mat_{d\times e}(\mathfrak O)$ and
$M'\subseteq \Mat_{d'\times e'}(\mathfrak O)$, we regard $M\oplus M'$ as a
submodule of $\Mat_{(d+d')\times (e+e')}(\mathfrak O)$, embedded
diagonally---we will refer to this as a \emph{diagonal direct sum}.
Then $\Zeta^{\ask}_{M\oplus M'}(Y) = \Zeta^{\ask}_{M}(Y) *_Y
  \Zeta^{\ask}_{M'}(Y)$.
Similarly, given (linear) group schemes $\mathsf G$ and $\mathsf G'$ over
$\mathfrak O$,
embedded into $d\times d$ and $d'\times d'$ matrices, respectively,
we embed $\mathsf G\times \mathsf G'$ diagonally into $(d+d')\times (d+d')$
matrices.
We have
$\Zeta^{\oc}_{\mathsf G\times \mathsf G'}(Y) = \Zeta^{\oc}_{\mathsf
    G}(Y) *_Y \Zeta^{\oc}_{\mathsf G'}(Y)$
and
$\Zeta^{\cc}_{\mathsf G\times \mathsf G'}(Y) = \Zeta^{\cc}_{\mathsf G}(Y) *_Y
  \Zeta^{\cc}_{\mathsf G'}(Y)$.
Finally, given graphs $\mathsf\Gamma$ and $\mathsf\Gamma'$, we have
$\Zeta^{\ask}_{\mathsf\Gamma \oplus \mathsf\Gamma'}(Y) = \Zeta^{\ask}_{\mathsf\Gamma}(Y) *_Y
  \Zeta^{\ask}_{\mathsf\Gamma'}(Y)$; moreover, $\mathsf G_{\mathsf\Gamma\oplus \mathsf\Gamma'} \cong
  \mathsf G_{\mathsf\Gamma} \times \mathsf G_{\mathsf\Gamma'}$.

In this way, explicit formulae and algorithms for computing Hadamard products
of rational generating functions become applicable to symbolic enumeration
problems in algebra.

%%%%%%%%%%%%%%%%%%%%%%%%%%%%%%%%%%%%%%%%%%%%%%%%
\subsection{Zeta functions from labelled coloured configurations}
\label{ss:lcc_zeta}
%%%%%%%%%%%%%%%%%%%%%%%%%%%%%%%%%%%%%%%%%%%%%%%%

It turns out that each zeta function from
Examples~\ref{ex:ask_mat}--\ref{ex:more_graphs}
can be expressed in terms of the rational functions $W_{f,\alpha}^\ve(X,Y)$
attached to labelled coloured configurations as in Section~\ref{elementary}.
In the following table, we write $\underline n$ for the sum of all $2^n$
coloured permutations of the form $1^{\nu_1} \dotsb n^{\nu_n}$
with $\nu_i\in \{0,i\}$.
A ``\dup{}'' indicates that an entry coincides with the one immediately
above it.

\begin{table}[h]
  \centering
  \begin{tabular}{|c|c|c|c|c|c|}
    \hline
    {Zeta function}                                                                                                                                               & $f$                    & $\alpha$                         & $\ve$          & $u(X)$                             & $W_{f,\alpha}^\ve(X,Y)$                                         \\ \hline
    %%%%%%%%%%%%%%%%%%%%%%%%%%%%%%%%%%%%%%%%%%%%%%%%
    $\Zeta^{\ask}_{\Mat_{d\times e}(\mathfrak O)}(Y)$                                                                                                             & $\underline 1$         & $1 \gets -X^{-d}$                & \small $d - e$ & $1$                                & $\frac{1-X^{-e}Y}{(1-Y)(1-X^{d-e}Y)}$                           \\ \hline
    %%%%%%%%%%%%%%%%%%%%%%%%%%%%%%%%%%%%%%%%%%%%%%%%
    \small
    \begin{tabular}{c} $\Zeta^{\ask}_{\mathfrak{so}_d(\mathfrak O)}(Y)$, \\$\Zeta^{\ask}_{\Mat_{d\times(d-1)}(\mathfrak O)}(Y)$\end{tabular} & \dup                   & \dup                             & $1$            & \dup                               & $\frac{1-X^{1-d}Y}{(1-Y)(1-XY)}$                                \\ \hline
    %%%%%%%%%%%%%%%%%%%%%%%%%%%%%%%%%%%%%%%%%%%%%%%%
    $\Zeta^{\cc}_{\mathsf F_{2,d}}(Y)$                                                                                                                            & \dup                   & \dup                             & \dup           & $X^{\binom d 2}$                   & \dup                                                            \\ \hline
    %%%%%%%%%%%%%%%%%%%%%%%%%%%%%%%%%%%%%%%%%%%%%%%%
    $\Zeta^{\ask}_{\mathsf\Delta_n\vee \mathsf K_{n+1}}(Y)$                                                                                                              & \underline 2           & \small $1,2\gets -X^{-n-1}$      & $1$            & $X^{-1}$                                & \small $\frac{(1-X^{1-n}Y)(1-X^{-n}Y)}{(1-Y)(1-XY)(1-X^2Y)}$ \\ \hline
    %%%%%%%%%%%%%%%%%%%%%%%%%%%%%%%%%%%%%%%%%%%%%%%%
    $\Zeta^{\cc}_{\mathsf G_{\mathsf\Delta_n\vee \mathsf K_{n+1}}}(Y)$
                                                                                                                                                                  &
                                                                                                                                                                    \dup
                                                                                                                                                                                           &
                                                                                                                                                                                             \dup                             & \dup           & \small $X^{3\binom{n+1} 2-1}$ & \dup                                                            \\ \hline
    %%%%%%%%%%%%%%%%%%%%%%%%%%%%%%%%%%%%%%%%%%%%%%%%
    $\Zeta^{\ask}_{\mathsf T_n}(Y)$ &
    \underline 4 &
    \small
    \begin{tabular}{c} $1,2\gets -X^{-n-3}$ \\
    $3,4 \gets -X^{-n-2}$ \end{tabular} & $1$            & $X^{-3}$
                                                                                                                                                                                                                              & \small $\frac{(1-X^{-n-1}Y)(1-X^{-n}Y)^2 (1-X^{1-n}Y)}{\prod_{i=0}^4 (1-X^iY)}$ \\ \hline
    %%%%%%%%%%%%%%%%%%%%%%%%%%%%%%%%%%%%%%%%%%%%%%%%

    $\Zeta^{\cc}_{\mathsf T_n}(Y)$
                                                                                                                                                                  &
                                                                                                                                                                    \dup
                                                                                                                                                                                           &
                                                                                                                                                                                            \dup                             & \dup           & \small $X^{5(n+3)(n+1)}$ & \dup                                                            \\ \hline
    %%%%%%%%%%%%%%%%%%%%%%%%%%%%%%%%%%%%%%%%%%%%%%%%
    $\Zeta^{\oc}_{\mathsf U_{d+1}}(Y)$                                                                                                                            & \small $\underline{d}$ & \small $1,\dotsc,d\gets -X^{-1}$ & 0              & $X$                                & $\frac{(1-X^{-1}Y)^{d}}{(1-Y)^{d+1}}$                           \\\hline
  \end{tabular}
  \caption{Examples of zeta functions from labelled coloured configurations}
  \label{tab}
\end{table}

\begin{pro}
  \label{prop:zeta_lcc}
  Let $\mathfrak O$ be a compact discrete valuation ring with residue field
  size $q$.
  Let $\Zeta(Y)$ be one of the rational generating functions in the first
  column of Table~\ref{tab}.
  We make the following additional assumptions:
  \begin{itemize}
    \item
          If $\Zeta(Y) = \Zeta^{\ask}_{\mathfrak{so}_d(\mathfrak O)}(Y)$, then we
          assume that $\mathfrak O$ has characteristic distinct from $2$.
    \item
          If $\Zeta(Y) = \Zeta^{\cc}_{\mathsf F_{2,d}}(Y)$, then we assume that $q$ is
          odd.
    \item
          If $\Zeta(Y) = \Zeta^{\oc}_{\mathsf U_d}(Y)$, then we assume
          that $\gcd(q,(d-1)!) = 1$.
  \end{itemize}
  Let $f$, $\alpha$, $\varepsilon$, and $u(X)$ be as in the corresponding row
  of Table~\ref{tab}.
  Then $$\Zeta(Y) = W_{f,\alpha}^\ve(q,u(q) Y).$$
\end{pro}
\begin{proof}
  These are straightforward calculations based on the explicit formulae
  listed above.
  We verify two cases and leave the others to the reader.

  \begin{itemize}
    \item $\Zeta(Y) = \Zeta_{\Mat_{d\times e}(\mathfrak O)}^{\ask}(Y)$:
          in this case, using Example~\ref{ex:1^0+1^1} and Example~\ref{ex:ask_mat},
          we find that
          $$
            W_{\underline 1, \alpha}^\varepsilon(q,Y) =
            \left(\frac{1 - X^{-d+\varepsilon} Y}{(1-Y)(1-X^\varepsilon
                Y)}\right)(X\gets q)
            = \Zeta_{\Mat_{d\times e}(\mathfrak O)}^{\ask}(Y).
          $$
    \item $\Zeta(Y) = \Zeta_{\mathsf U_d}^{\oc}(Y)$:
          by Example~\ref{ex:oc_Ud}, we have
          $\Zeta^{\oc}_{\mathsf U_{d+1}}(q^{-1}Y) =
            \frac{(1-q^{-1}Y)^{d}}{(1-Y)^{d+1}}$.
          A coloured permutation $\bm a = 1^{\nu_1}\dotsb d^{\nu_d}$
          with $\nu_i\in \{0,i\}$ is uniquely determined by its
          descent set
          $J = \Des(\bm a) = \{ i\in \{0,\dotsc,d-1\} : \nu_{i+1}\not= 0\}$.
          For the given function $\alpha$, we then have $\alpha(\bm a) = (-X)^{-\abs J}$.
          The claim follows since
          $$
            W_{\underline d,\alpha}^0(X,Y) =
            \frac{\sum_{J\subseteq \{0,\dotsc,d-1\}} (-X)^{-\abs J} Y^{\abs J}}
            {(1-Y)^{d+1}} = \frac{(1-X^{-1}Y)^d}{(1-Y)^{d+1}}
          $$
          by the Binomial Theorem.
          \qedhere
  \end{itemize}
\end{proof}

%%%%%%%%%%%%%%%%%%%%%%%%%%%%%%%%%%%%%%%%%%%%%%%%
\subsection{Applications}
\label{ss:app}
%%%%%%%%%%%%%%%%%%%%%%%%%%%%%%%%%%%%%%%%%%%%%%%%

Let $\mathfrak O$ be a compact discrete valuation ring with residue field size
$q$.
We now explain how, subject to a compatibility condition,
Theorem~\ref{thm:main_elementary} provides a method for explicitly computing Hadamard
products of the zeta functions in Table~\ref{tab}.
As explained in Section~\ref{ss:ask}, we can interpret such Hadamard products as zeta
functions associated with suitable products of the objects under consideration.
We first record an elementary observation.

\begin{lem}
  \label{lem:shift}
  Let $R$ be a ring.
  Let $A(Y) = \sum_{k=0}^\infty a_k Y^k$ and $B(Y) = \sum_{k=0}^\infty b_k
    Y^k$ be formal power series over $R$.
  Let $u,v\in R$.
  Then $A(uY) *_Y B(vY) = (A *_Y B)(uvY)$. \qed
\end{lem}

As in Section~\ref{ss:lcc}, let $\mathbb U = \{ \pm X^{k} : k\in \ZZ\}$.
By combining Theorem~\ref{thm:main_elementary} and Lemma~\ref{lem:shift}, we
obtain the following.

\begin{cor}
  \label{cor:same_eps}
  Let $(f,\alpha)$ and $(g,\beta)$ be labelled coloured
  configurations.
  Let $\varepsilon \in \ZZ$ and let $u(X),v(X) \in \mathbb U$.
  Suppose that $f$ and $g$ are strongly disjoint.
  Then
  $$W_{f,\alpha}^\ve\bigl(X,\, u(X)Y\bigr) *_Y W_{g,\beta}^\ve\bigl(X,\, v(X)
  Y\bigr)
  = W_{f\shuffle g,\alpha\beta}^\ve \bigl(X,\,u(X)v(X)Y\bigr).
  \eqno\qed$$
\end{cor}

As detailed in Section~\ref{elementary}
(and explained in terms of $(\des,\comaj, \bcol)$-equivalence in
Section~\ref{s:coloured}),
given labelled coloured configurations $(f,\alpha)$ and $(g,\beta)$,
for the purpose of computing
$W_{f,\alpha}^\ve\bigl(X,\, u(X)Y\bigr) *_Y W_{g,\beta}^\ve\bigl(X,\, v(X)
  Y\bigr)$,
we may always pass to equivalent labelled coloured configurations such that
$f$ and $g$ are strongly disjoint.
(We will tacitly do so in the proofs below.)
The compatibility condition that we alluded to above is that we require
entries in the $\ve$-column of Table~\ref{tab} to agree for us to compute
associated Hadamard products via Corollary~\ref{cor:same_eps}.
All that remains to obtain our zeta function is then to specialise $X\gets q$.

We now record what we regard as the most appealing and interesting
applications of Theorem~\ref{thm:main_elementary} in the context of
Proposition~\ref{prop:zeta_lcc} and Table~\ref{tab}.
First, given a set of (uncoloured) permutations $P$, we write
$$\Pi(P)
  = \left\{
  \sigma_1^{\gamma_1}\dotsb \sigma_n^{\gamma_n} :
  \sigma_1\dotsb\sigma_n\in P, \, \gamma_i \in \{0,\sigma_i\} \text{ for } i = 1,\dotsc,n
  \right\}.$$
As usual, we write $\mathrm S_n = \{ \sigma_1\dotsb \sigma_n : \{
  \sigma_1,\dotsc,\sigma_n\} = [n]\}$ for the symmetric group on $n$ points.

For a function $\alpha\colon
  \Gamma \to \mathbb U$ and coloured permutation $\bm a$ as in Section~\ref{ss:lcc},
we write $\alpha_q(\bm a) = (\alpha(\bm a))(X\gets q)$.

\begin{cor}
  \label{cor:Hadamard_Mde}
  Let $d_1,\dotsc,d_n,e_1,\dotsc,e_n$ be such that $d_i - e_i$ is independent
  of $i$.
  Let $\delta := d_1 - e_1 = \dotsb = d_n - e_n$ denote this common value.
  Let $\alpha\colon \Gamma \to \mathbb U$ with $\alpha(i) = -X^{-d_i}$ for
  $i=1,\dotsc,n$ and $\alpha(c) = 1$ otherwise.
  Then
  $$\Zeta^{\ask}_{\Mat_{d_1\times e_1}(\mathfrak O) \oplus \dotsb \oplus
    \Mat_{d_n\times e_n}(\mathfrak O)}(Y) =
    \frac{\sum_{\bm a \in \Pi(\mathrm S_n)} \alpha_q(\bm a) q^{\delta \comaj(\bm
        a)}Y^{\des(\bm a)}}
    {(1-Y)(1-q^\delta Y)\dotsb (1-q^{n \delta}Y)}.
  $$
\end{cor}
\begin{proof}
  Let $\alpha_i\colon \Gamma \to \mathbb U$ with $\alpha_i(i) = -X^{-d_i}$ and
  $\alpha(c) = 1$ otherwise; note that $\alpha = \alpha_1\dotsb
    \alpha_n$.
  By Proposition~\ref{prop:zeta_lcc}, we have
  $\Zeta^{\ask}_{\Mat_{d_i\times e_i}(\mathfrak O)}(Y) =
    W^{\delta}_{i^0 + i^i, \alpha_i}(q,Y)$.
  By Theorem~\ref{thm:main_elementary}, we thus have
  $$\Zeta^{\ask}_{\Mat_{d_1\times e_1}(\mathfrak O) \oplus \dotsb \oplus
    \Mat_{d_n\times e_n}(\mathfrak O)}(Y)
    = \BigAsterisk_{i=1}^n \Zeta^{\ask}_{\Mat_{d_i\times e_i}(\mathfrak O)}(Y)
    = W^\delta_{(1^0+1^1) \shuffle \dotsb \shuffle (n^0 + n^n), \alpha}(q,Y).
  $$
  The claim follows since the iterated shuffles in question are
  precisely the elements of $\Pi(\mathrm S_n)$.
\end{proof}

We note that Corollary~\ref{cor:Hadamard_Mde} completely
answers \cite[Question~10.4(a)]{cico} and partially answers
\cite[Question~10.4(b)]{cico}.
What remains elusive is an explicit formula for the Hadamard products in
Corollary~\ref{cor:Hadamard_Mde} \itemph{without} our assumption that $d_1 - e_1 =
  \dotsb = d_n - e_n$.

\begin{rmk}
  Given \itemph{arbitrary} $d_i$ and $e_i$, a formula for
  $$\Zeta^{\ask}_{\Mat_{d_1\times e_1}(\mathfrak O) \oplus \dotsb \oplus
    \Mat_{d_n\times e_n}(\mathfrak O)}(Y) =
    \Zeta^{\ask}_{\Mat_{d_1\times e_1}(\mathfrak O)} *_Y \dotsb *_Y
    \Zeta^{\ask}_{\Mat_{d_n\times e_n}(\mathfrak O)}(Y)
  $$
  as a sum of
  $\Theta\bigl(\frac{n!}{(\log 2)^{n}}\bigr)$ (``Big Theta Notation'')
  explicit rational functions appears in \cite[Cor.~5.15]{cico}.
  It, however, remains unclear how to derive a meaningful expression of these
  Hadamard products as quotients of explicit combinatorially-defined polynomials.
  In case $d_1 - e_1 = \dotsb = d_n - e_n$,
  Corollary~\ref{cor:Hadamard_Mde} does just that.
\end{rmk}

\begin{rmk}
  \label{r:brenti}
  There is an evident bijection between $\Pi(\mathrm S_n)$ and the group
  $\mathrm B_n$ of signed permutations on $n$ points (of order $2^n n!$).
  With this identification, the special case $d_1 = e_1 = \dotsb = d_n = e_n$ of
  Corollary~\ref{cor:Hadamard_Mde} (essentially) coincides with \cite[Prop.\
    10.3]{cico} and its close relative \cite[Cor.\ 5.17]{Rossmann/18}
  (which is the case $d_1 = e_1 = \dotsb = d_n = e_n = 1$).
  To our knowledge, prior to the present work, these were the only known examples of
  Hadamard products of ask (as well as class- and orbit-counting) zeta functions
  expressed in terms of (coloured) permutation statistics.
  In the aforementioned results in the literature, the proofs rely on work of
  Brenti~\cite{Bre94}.
  These proofs are based on the coincidence of the rational generating
  functions in question with those attached to so-called $q$-Eulerian
  polynomials of signed permutations.
  This preceded more recent machinery surrounding shuffle
  compatibility.
  In particular, this earlier work is now explained as part of the framework presented here.
\end{rmk}

The following is a group-theoretic application of
Corollary~\ref{cor:Hadamard_Mde}.

\begin{cor}
  \label{cor:hada_F2d}
  Let $d_1,\dotsc,d_n \ge 1$.
  Define $\alpha$ as in Corollary~\ref{cor:Hadamard_Mde}.
  Suppose that $q$ is odd.
  Then
  $$\Zeta^{\cc}_{\mathsf F_{2,d_1} \times \dotsb \times \mathsf
    F_{2,d_n}}(q^{-\sum_{i=1}^n \binom{d_i} 2} Y)
    = \frac{\sum_{\bm a \in \Pi(\mathrm S_n)} \alpha_q(\bm a) q^{\comaj(\bm
        a)}Y^{\des(\bm a)}}
    {(1-Y)(1-q Y)\dotsb (1-q^{n}Y)}.
  $$
\end{cor}
\begin{proof}
  This follows since $\Zeta^{\cc}_{\mathsf F_{2,d}}(Y) =
    \Zeta^{\ask}_{\Mat_{d\times(d-1)}(\mathfrak O)}(q^{\binom d 2} Y)$.
\end{proof}

We can also symbolically compute the orbit-counting zeta function of $\mathsf
  U_{d_1+1}\times \dotsb \times \mathsf U_{d_n+1}$.

\begin{cor}
  \label{cor:hada_Ud}
  Given $d_1,\dotsc,d_n\ge 0$, write
  $D_i = d_1 + \dotsb + d_i$ and
  $$
    T = (D_0+1) \dotsb D_1
    \shuffle
    \dotsb
    \shuffle
    (D_{n-1}+1)\dotsb D_n
    \subseteq \mathrm S_{D_n},
  $$
  a set of uncoloured permutations of cardinality $\prod_{i=1}^n
    \binom{D_{i}}{D_{i-1}}$.
    Recall the definition of $\palette^*$ from Section~\ref{ss:colperm}.
  Let $m = \max\limits_{i=1,\dotsc,n} d_i$.
  Suppose that $\gcd(q, (m - 1)!) = 1$.
  Then
  \vspace*{-1em}
  $$
    \Zeta^{\oc}_{\mathsf U_{d_1+1}\times \dotsb\times \mathsf U_{d_n+1}}(q^{-n}Y) =
    \BigAsterisk_{i=1}^n \Zeta^{\oc}_{\mathsf U_{d_i+1}}(q^{-1}Y) =
    \frac{\sum_{\bm a\in \Pi(T)}
    (-q)^{-\abs{\palette^*(\bm a)}}\,
    Y^{\des(\bm a)}}
    {(1-Y)^{D_n+1}}.
  $$
\end{cor}
\begin{proof}
  Define $\alpha\colon \Gamma \to \mathbb U$ with $\alpha(j) = -X^{-1}$ for
  $j = 1,\dotsc,D_n$ and $\alpha(c) = 1$ otherwise.
  Let $\alpha_i\colon \Gamma \to \mathbb U$ with $\alpha_i(j) = -X^{-1}$
  for $j=D_{i-1}+1,\dotsc,D_i$ and $\alpha_i(c) = 1$ otherwise;
  note that $\alpha = \alpha_1\dotsb \alpha_n$.
  By construction, for $\bm a \in \Pi(T)$, we have $\alpha_q(\bm a) = (-q)^{-\abs{\palette^*(\bm a)}}$.
  Extending our previous notation, for $a\le b$, write
  $\underline{a\dotsc b}$ for the sum over all $a^{\nu_a} \dotsb b^{\nu_b}$
  with $\nu_i$ satisfying $\nu_i \in \{0,i\}$.
  Then
  \begin{equation}
    \label{eq:Ud_Hadamard}
    \BigAsterisk_{i=1}^n \Zeta^{\oc}_{\mathsf U_{d_i+1}}(q^{-1}Y)
    =
    \BigAsterisk_{i=1}^n W^0_{\underline{D_{i-1}+1 \dotsc D_i},\alpha_i}(q,Y)
    =
    W^0_{\underline{D_0+1\dotsc D_1} \shuffle \dotsb \shuffle
    \underline{D_{n-1}+1\dotsc D_n},\alpha}(q,Y);
  \end{equation}
  the $\gcd$ condition on $q$ is justified by \cite[Cor.\ 8.16]{Rossmann/18}
  and the fact that $\mathsf U_{d_1}\times \dotsb \times \mathsf U_{d_r}$ has
  nilpotency class $\max_{i=1,\dotsc,n}d_i - 1$.
  The claim follows since the iterated shuffles in \eqref{eq:Ud_Hadamard} are
  precisely the elements of $\Pi(T)$.
\end{proof}

%%%%%%%%%%%%%%%%%%%%%%%%%%%%%%%%%%%%%%%%%%%%%%%%
\subsection*{Acknowledgements}
%%%%%%%%%%%%%%%%%%%%%%%%%%%%%%%%%%%%%%%%%%%%%%%%

AC and TR gratefully acknowledge interesting discussions with Christopher Voll.
AC was partially supported by the College of Science and Engineering at the
University of Galway through a Strategic Research (Millennium) Fund.
We are grateful to the anonymous referees for carefully reading the
manuscript and for providing valuable suggestions.

\bibliographystyle{abbrv}
\bibliography{hadamard.bib}

\end{document}